\documentclass[11pt]{article}
\usepackage{latexsym}
\usepackage{amssymb}
\usepackage{amsmath}
\usepackage{mathrsfs}
\usepackage{enumerate}
\usepackage{bm}
\usepackage{graphics}
\usepackage{graphicx}
\usepackage{dsfont}
\usepackage{subfigure}
\usepackage{color}
\usepackage{epstopdf}
\usepackage[T1]{fontenc}
\usepackage{latexsym,amssymb,amsmath,amsfonts,amsthm}\usepackage{txfonts}
\usepackage{amsfonts,amsthm}

\newcommand{\R}{{\mathbb R}}

\newcommand{\N}{{\mathbb N}}

\newcommand{\no}{\nonumber}
\newcommand{\be}{\begin{eqnarray}}
\newcommand{\ben}{\begin{eqnarray*}}
\newcommand{\en}{\end{eqnarray}}
\newcommand{\enn}{\end{eqnarray*}}

\newcommand{\ov}{\overline}

\newcommand{\G}{\Gamma}

\newcommand{\Om}{\Omega}
\newcommand{\om}{\omega}

\newtheorem{theorem}{Theorem}[section]

\newtheorem{remark}[theorem]{Remark}

\definecolor{ljl}{rgb}{0,0,0}

\allowdisplaybreaks[4]
\begin{document}
\renewcommand{\theequation}{\arabic{section}.\arabic{equation}}
%\begin{titlepage}
\title{\bf The Nystr\"{o}m method for elastic wave scattering by unbounded rough surfaces 
%The numerical algorithm of elastic wave scattering by unbounded rough surfaces
}

\author{Jianliang Li\thanks{School of Mathematics and Statistics, Changsha University of Science
and Technology, Changsha, 410114, China (\tt lijl@amss.ac.cn)}
\and
Xiaoli Liu\thanks{Corresponding author. INRIA and Ecole Polytechnique (CMAP), 
Institut Polytechnique de Paris, Route de Saclay, 91128, Palaiseau, France
({\tt xiaoli.liu@inria.fr})}
\and
Bo Zhang\thanks{LSEC, NCMIS and Academy of Mathematics and Systems Science, Chinese Academy of Sciences, 
Beijing 100190, China, and School of Mathematical Sciences, University of Chinese Academy of Sciences, 
Beijing 100049, China ({\tt b.zhang@amt.ac.cn})}
\and
Haiwen Zhang\thanks{NCMIS and Academy of Mathematics and Systems Science, Chinese Academy of Sciences,
Beijing 100190, China, and School of Mathematical Sciences, University of Chinese Academy of Sciences, 
Beijing 100049, China ({\tt zhanghaiwen@amss.ac.cn})}}
\date{}
%\end{titlepage}
\maketitle

\begin{abstract}
We consider the numerical algorithm for the two-dimensional time-harmonic elastic wave scattering by unbounded rough surfaces with Dirichlet boundary condition. A Nystr\"{o}m method is proposed for the scattering problem 
based on the integral equation method.
Convergence of the Nystr\"{o}m method is established with convergence rate depending on the smoothness of 
the rough surfaces. In doing so, a crucial role is played by analyzing the singularities of the kernels 
of the relevant boundary integral operators. Numerical experiments are presented to demonstrate the 
effectiveness of the method.

\vspace{.2in}
\end{abstract}

\section{Introduction}\label{se1}
\setcounter{equation}{0}

We consider the two-dimensional time-harmonic elastic scattering problem for unbounded rough surfaces 
with Dirichlet boundary condition. This kind of problem has attracted a lot of attentions over the 
last decade since it has a wide range of applications in diverse scientific areas such as seismology 
and nondestructive testing.

%The direct scattering problem consists of two components. One is the well-posedness, which is to prove the
%uniqueness, existence and continuous dependence of scattered wave when the incident wave, the rough surface 
%along with its boundary condition are given; the other is the numerical solution, which aims to develop 
%accurate, fast and stable algorithms to numerically solve the scattered wave.

The well-posedness of the direct scattering problems by rough surfaces has been extensively studied 
in the past thirty years. For acoustic case, the authors in \cite{SC99,SCZ99,SZ982,R96} first employed 
integral equation methods to prove that the scattering problem by an infinite rough surface with 
Dirichlet boundary condition is uniquely solvable.
These works have been extended to other boundary conditions, see \cite{ZS03} for the impedance case 
and \cite{dtS03} for  the penetrable case. By using variational formulation,
the unique solvability of the sound-soft rough surface scattering problem has been established 
in \cite{SE10,SM05}. For the case of elastic scattering by Dirichlet rough surfaces,
the uniqueness result was proved in \cite{TA01}, while the existence result was given in \cite{TA02} 
using the boundary integral equation method (see also \cite{TA} for a comprehensive discussion). 
The authors in \cite{JG12} studied the well-posedness of the elastic scattering problem by unbounded 
rough surfaces via the variational approach.

The computational aspect of the scattering problem by bounded obstacles has been extensively studied 
(see, e.g., \cite{I98,Monk2003} for  the finite elements method based on variational formulation 
and \cite{CK19,Tc07} for the Nystr\"{o}m method based on boundary integral equations).
For the acoustic scattering by unbounded rough surfaces,
many numerical algorithms have been already developed.
The numerical approach using finite elements method, combined with the perfectly matched layer technique, 
was presented in \cite{SM09}.
By dealing with a class of integral equations on the real line, the Nystr\"{o}m method has been applied 
to the acoustic scattering problem for the sound-soft rough surfaces \cite{MACK,LB13} and for the 
penetrable rough surfaces \cite{LYZ13}. However, few works are available for the numerical solution 
of elastic scattering by unbounded rough surfaces.

In this paper, we propose the Nystr\"{o}m method for two-dimensional time-harmonic elastic scattering 
problem for unbounded rough surfaces with Dirichlet boundary condition.
Our method is based on the integral equation formulations given in \cite{TA,TA02}, which can be reduced 
to a class of integral equations on the real line. A crucial role of our method is played by a thorough 
analysis on the singularities of the  kernels in the relevant integral equations, which involves the 
Green tensor for Navier equation in the half-space. By spliting off the logarithmic singularity in 
the related kernels and using the asymptotic behavior of the Bessel functions, we obtain the 
convergence of the Nystr\"{o}m method with convergence rate depending on the smoothness of the 
rough surfaces. Several numerical examples are presented to verify our theoretical results and 
show the effectiveness of our method.

The paper is organized as follows. In Section \ref{sec2}, we give a brief introduction on the 
mathematical model of the scattering problem and present the existed well-posedness result using 
the integral equation method. Section \ref{sec3}
is devoted to analyzing the singularities for the relevant kernels included in the integral impression of the solution. 
In Section \ref{sec4}, we establish the convergence of the Nystr\"{o}m method. Numerical experiments 
are given to show the effectiveness of the proposed method in Section \ref{sec5}. 
Finally, we give a conclusion in Section \ref{sec6}.

\section{The well-posedness of the scattering problem}\label{sec2}\setcounter{equation}{0}

In this section, we present the existed results for the well-posedness of the two-dimensional elastic 
wave scattering problem by unbounded rough surfaces.
Fisrt, we introduce some basic notations and function spaces used in this paper. For $V\subset\R^m, m=1,2$, 
let $BC(V)$ represent the set of bounded and continuous complex-valued functions on $V$ under the norm $\|\varphi\|_{\infty, V}:=\sup_{x\in V}|\varphi(x)|$. We denote by $BC^n(\R^m)$ the set of all functions 
whose derivatives up to order $n$ are bounded and continuous on $\R^m$ with the norm
\begin{eqnarray*}
\|\varphi\|_{BC^n(\R^m)}:=\max_{l=0,1,...,n}\max_{|\alpha|=l}\|\partial_{x_1}^{\alpha_1}
\partial_{x_2}^{\alpha_2}...\partial_{x_m}^{\alpha_m}\varphi\|_{\infty,\R^m},
\end{eqnarray*}
where $\alpha=(\alpha_1,\alpha_2,....,\alpha_m)$ and $|\alpha|:=\sum_{i=1}^m\alpha_i$. We define 
$C_{0,\pi}^n(\R^2)$ and $BC_p^n(\R^2)$ as
\begin{eqnarray*}\label{b2}
&&C_{0,\pi}^n(\R^2):=\{a(s,t)\in BC^n(\R^2): a(s,t)=0\;\;{\rm for }\; |s-t|\geq \pi\},\\ \label{b3}
&&BC_p^n(\R^2):=\{a(s,t)\in BC^n(\R^2): \|a\|_{BC_p^n(\R^2)}<\infty\}
\end{eqnarray*}
with the norm
\begin{eqnarray*}
\|a\|_{BC_p^n(\R^2)}:=\sup_{j,k=0,...,n, j+k\leq n}\left\|\widetilde{w}_p(s,t)
\frac{\partial^{j+k}a(s,t)}{\partial^j s\partial^k t}\right\|_{\infty},
\end{eqnarray*}
where the weight $\widetilde{w}_p(s,t):=(1+|s-t|)^p$ for some $p>1$, which are closed subspaces 
of $BC^n(\R^2)$. Let $H^1(V)$ and $H^{1/2}(\partial V)$ be the standard Sobolev spaces for any 
open set $V\subset\R^m$ provided the boundary of $V$ is smooth enough. The notations $H^1_{\rm loc}(V)$ 
and $H^{1/2}_{\rm loc}(V)$ stand for the set of functions which are elements of $H^1(\mathcal{V})$ 
and $H^{1/2}(\mathcal{V})$ for any $\mathcal{V}\subset\subset V$, respectively. Here the notation $\mathcal{V}\subset\subset V$ denotes that the closure of $\mathcal{V}$ is a compact subset of $V$.

Throughout this paper, let $h$ be a real number with $h<\inf_{x_1\in\R}f(x_1)$, where $f$ is the 
function of the rough surface which will be introduced later. We define the half-plane $U_h$ and 
its boundary $\Gamma_h$ as
\begin{eqnarray*}
U_h:=\{x\in\R^2: x_2>h\}\quad{\rm and}\quad \Gamma_h:=\{x\in\R^2: x_2=h\}.
\end{eqnarray*}
For $y=(y_1,y_2)\in U_h$, $y'$ is defined as
\begin{eqnarray*}
y':=(y_1,2h-y_2).
\end{eqnarray*}
The notations $J_n$ and $Y_n$ are Bessel functions and Neumann functions of order $n$, respectively. 
The linear combination
\begin{eqnarray*}
H_n^{(1)}:=J_n+{\rm i}Y_n
\end{eqnarray*}
is known as the Hankel function of the first kind of order $n$.

\begin{figure}[htbp]
\centering
\includegraphics[width=0.65\linewidth]{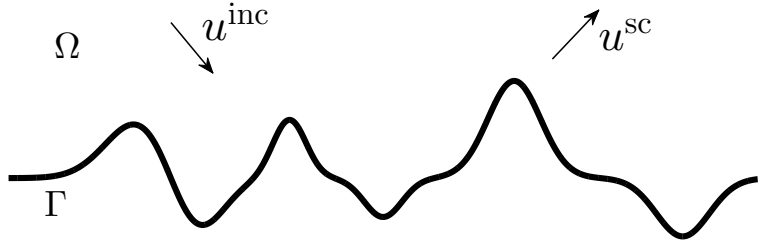}
\caption{Physical configuration of the scattering problem by a rough surface}
\label{figure1}
\end{figure}

As shown in Figure \ref{figure1}, the rough surface $\Gamma$ is described as
\ben
\Gamma:=\left\{(x_1,x_2)\in\R^2: x_2=f(x_1)\right\},
\enn
where $f\in  B^{(n)}_{c,M}$ with
\ben
B^{(n)}_{c,M}:=\{f\in BC^{n+2}(\R):\|f\|_{BC^{n+2}(\R)}\leq M\;\;{\rm and }\;\;\inf_{x_1\in\R}f(x_1)\geq c\}
\enn
for some nonnegative integer $n$, some constants $c>h$ and $M>0$, which implies that the 
surface $\Gamma$ lies above the $x_1$-axis.
The whole space is separated by $\Gamma$ into two unbounded half-spaces and the upper one is denoted by
\ben
\Omega:=\{(x_1,x_2)\in\R^2: x_2>f(x_1)\}.
\enn
Suppose an incoming field $\bm{u}^{\rm inc}=(u_1^{\rm inc}, u_2^{\rm inc})$
is incident on the infinite surface $\Gamma$ from the upper region $\Omega$. Then the scattering of 
$\bm{u}^{\rm inc}$ by the infinite rough surface $\Gamma$ can be modeled by the two-dimensional 
Navier equation
\begin{equation}\label{a1}
\mu \Delta {\bm u}^{\rm sc}+(\lambda + \mu)\nabla\nabla\cdot{\bm u}^{\rm sc}
+\omega^2{\bm u}^{\rm sc}=0\quad{\rm in}\;\;\Omega,
\end{equation}
where $\bm{u}^{\rm sc} =(u_1^{\rm sc}, u_2^{\rm sc})$ is the scattered field, $\omega>0$ represents the angular frequency, and $\lambda$ and $\mu$ are the Lam\'{e} constants satisfying $\mu>0$ and $\lambda+2\mu>0$, which leads to that the second order partial differential operator $\Delta^*:=\mu\Delta+(\lambda + \mu)\nabla\nabla\cdot$ is strongly elliptic \cite{M}. Let $\bm{u}:=\bm{u}^{\rm inc}+\bm{u}^{\rm sc}$ denote the total field consisting of the incident field and the scattered field. Further, the Dirichlet boundary condition is imposed on $\Gamma$, that is,
\begin{equation}\label{a2}
\bm{u}^{\rm sc}=-\bm{u}^{\rm inc}\quad{\rm on}\;\;\Gamma.
\end{equation}

Since the Navier equation (\ref{a1}) is imposed in the unbounded region $\Omega$, an appropriate radiation condition is needed for the considered scattering problem. In this paper, the scattered wave is assumed to satisfy the following upwards propagating radiation condition (UPRC) \cite{TA02}:
\be\label{a3}
{\bm u}^{\rm sc}(x)=\int_{\Gamma_H}{\bm \Pi}_{D,H}^{(2)}(x,y)\bm{\phi}(y)ds(y),\quad x\in U_H,
\en
for some $\bm{\phi}\in \left[L^{\infty}(\Gamma_H)\right]^2$ with $H>\sup\limits_{x_1\in\R}f(x_1)$, 
and the kernel $\bm{\Pi}_{D,H}^{(2)}(x,y)=\left(\Pi_{D,H,jk}^{(2)}(x,y)\right)_{j,k=1,2}$ in (\ref{a3}) 
is a matrix function with the elements given by
\begin{equation}\label{a4}
\Pi_{D,H,jk}^{(2)}(x,y)=\left(\bm{P}^{(y)}\left(\bm{G}_{D,H,j\cdot}(x,y)\right)^\top\right)_k,
\end{equation}
where $\bm{P}$ is the generalized stress vector defined on a curve $\Lambda\in\R^2$ with $\nu$ being the unit normal on $\Lambda$, that is,
\begin{eqnarray}\label{a5}
\bm{P}\bm{\varphi}:=(\mu+\widetilde{\mu})\frac{\partial \bm{\varphi}}{\partial \nu}+\widetilde{\lambda}\nu{\rm div}\bm{\varphi}-\widetilde{\mu}\nu^{\perp}{\rm div}^{\perp}\bm{\varphi}
\end{eqnarray}
with $\widetilde{\lambda}$, $\widetilde{\mu}\in\mathbb{R}$ satisfying $\widetilde{\lambda}+\widetilde{\mu}=\lambda+\mu$, $\nu^{\perp}=(\nu_2,-\nu_1)^\top$ for $\nu=(\nu_1,\nu_2)^\top$, and ${\rm div}^{\perp}\bm{\varphi}:=\frac{\partial\varphi_1}{\partial x_2}-\frac{\partial\varphi_2}{\partial x_1}$ for $\bm{\varphi}=(\varphi_1,\varphi_2)^\top$.
In \eqref{a4}, $\nu$ is the unit normal on $\Gamma_H$ pointing to the half-plane $U_H$ and $\bm{G}_{D,H}$ 
is the Green's tensor of the Navier equation (\ref{a1}) in  $U_H$ with 
Dirichlet boundary condition on $\Gamma_H$, which is given by
\begin{equation}\label{a6}
\bm{G}_{D,H}(x,y):=\bm{G}(x,y)-\bm{G}(x,y')+\bm{U}(x,y),\quad x,y\in U_H\quad{\rm and}\quad x\neq y,
\end{equation}
where $\bm{G}$ is the Green's tensor for the Navier equation (\ref{a1}) in free space $\mathbb{R}^2$, 
which is defined by
\begin{eqnarray}\label{a7}
\bm{G}(x,y)=\frac{1}{\mu}\Phi(x,y,\kappa_{\rm s})\bm{I}
+\frac{1}{\omega^2}\nabla_x\nabla_x^\top\left(\Phi(x,y,\kappa_{\rm s})
-\Phi(x,y,\kappa_{\rm p})\right),\quad x,y\in\R^2\quad{\rm and}\quad x\neq y,
\end{eqnarray}
with $\kappa_{\rm s}$ and $\kappa_{\rm p}$  being the shear and compressional wavenumbers defined by
\ben
\kappa_{\rm s}:=c_{\rm s}\omega, \quad\kappa_{\rm p}:=c_{\rm p}\omega \quad\text{with}\quad c_{\rm s}=\mu^{-1/2}, \quad c_{\rm p}:=(\lambda+2\mu)^{-1/2},
\enn
and
\ben
\Phi(x,y,\kappa):=\frac{\rm i}{4}H_0^{(1)}(\kappa |x-y|)
\enn
being the fundamental solution for the two-dimensional Helmholtz equation, and $\bm{U}(x,y)$ in (\ref{a6}) is a matrix function defined by
\begin{eqnarray*}
\bm{U}(x,y)=-\frac{\rm i}{2\pi\omega^2}\int_{\mathbb R}\left[\bm{M}_{\rm p}(\tau, \gamma_{\rm p},\gamma_{\rm s}; x_2,y_2)+\bm{M}_{\rm s}(\tau, \gamma_{\rm p},\gamma_{\rm s}; x_2,y_2)\right]e^{-{\rm i}(x_1-y_1)\tau}d\tau
\end{eqnarray*}
with
\begin{eqnarray*}
	\gamma_{\rm p}:=\sqrt{\kappa_{\rm p}^2-\tau^2},\quad\gamma_{\rm s}:=\sqrt{\kappa_{\rm s}^2-\tau^2},
\end{eqnarray*}
and
\begin{eqnarray*}
&&\bm{M}_{\rm p}(\tau, \gamma_{\rm p},\gamma_{\rm s}; x_2,y_2):=\frac{e^{{\rm i}\gamma_{\rm p}(x_2+y_2-2h)-e^{{\rm i}(\gamma_{\rm p}(x_2-h)+\gamma_{\rm s}(y_2-h))}}}{\gamma_{\rm p}\gamma_{\rm s}+\tau^2}\left[
  \begin{array}{cc}
    -\tau^2\gamma_{\rm s} & \tau^3 \\
    \tau\gamma_{\rm p}\gamma_{\rm s} & -\tau^2\gamma_{\rm p} \\
  \end{array}
\right],\\
&&\bm{M}_{\rm s}(\tau, \gamma_{\rm p},\gamma_{\rm s}; x_2,y_2):=\frac{e^{{\rm i}\gamma_{\rm s}(x_2+y_2-2h)-e^{{\rm i}(\gamma_{\rm s}(x_2-h)+\gamma_{\rm p}(y_2-h))}}}{\gamma_{\rm p}\gamma_{\rm s}+\tau^2}\left[
  \begin{array}{cc}
    -\tau^2\gamma_{\rm s} & -\tau\gamma_{\rm p}\gamma_{\rm s} \\
    -\tau^3 & -\tau^2\gamma_{\rm p} \\
  \end{array}
\right].
\end{eqnarray*}
With the aid of \cite[Theorem 2.1]{TA02}, we have $\bm{U}(x,y)\in \left[C^{\infty}(U_H)\cap C^1(\overline{U_H})\right]^{2\times 2}$, which will be used in the analysis on the convergence of the Nystr\"{o}m method.
We refer to \cite{TA, TA01} for more properties of the UPRC, and its relation to the Rayleigh expansion radiation condition for diffraction grating and the Kupradze's radiation condition for the scattering by bounded obstacles.

To ensure the uniqueness of the scattering problem, we need the following vertical growth rate condition
\begin{eqnarray}\label{a8}
\sup_{x\in\Omega}|x_2|^{\beta}|\bm{u}^{\rm sc}(x)|<\infty\quad{\rm for\;some\;\beta\in\R}.
\end{eqnarray}

In summary, the scattering problem (\ref{a1})--(\ref{a3}) and (\ref{a8}) can be described by the following boundary value problem with $\bm{g}=-\bm{u}^{\rm inc}|_{\Gamma}$:

{\bf Dirichlet Problem} (DP): Given ${\bm g}\in \left[BC(\G)\cap H^{1/2}_{\rm loc}(\G)\right]^2$,
find $\bm{u}^{\rm sc}\in \left[C^2(\Om)\cap C(\ov{\Omega})\cap H^1_{\rm loc}(\Omega)\right]^2$ such that
\begin{enumerate}[(i)]
\item
$\bm{u}^{\rm sc}$  is a solution of the Navier equation (\ref{a1}) in $\Omega$,
\item
$\bm{u}^{\rm sc}$ satisfies the Dirichlet boundary condition $\bm {u}^{\rm sc} = \bm{g}$ on $\Gamma$,
\item
$\bm{u}^{\rm sc}$ satisfies the UPRC (\ref{a3}),
\item
$\bm{u}^{\rm sc}$ satisfies the vertical growth rate condition (\ref{a8}).
\end{enumerate}

The following uniqueness result has been proved in \cite[Theorem 4.6]{TA01} for the problem (DP).

\begin{theorem}(\cite[Theorem 4.6]{TA01})\label{thm1}
 The problem (DP) has at most one solution.
\end{theorem}

The existence of the solution to the problem (DP) has been investigated in \cite{TA02} by integral 
equation method. The main idea is to seek for a solution in the form of a combined single- and 
double-layer potential
\begin{equation}\label{a9}
{\bm u}^{\rm sc}(x)=\int_{\Gamma}\left[{\bm \Pi}_{D,h}^{(2)}(x,y)-{\rm i}\eta\bm{G}_{D,h}(x,y)\right]\bm{\psi}(y)ds(y),\quad x\in \Omega,
\end{equation}
where ${\bm \Pi}_{D,h}^{(2)}(x,y)$ is defined similar as
${\bm \Pi}_{D,H}^{(2)}(x,y)$ in \eqref{a3} with $\nu$ being the unit normal  on $\Gamma$ pointing to $\Omega$,
 $\bm{\psi}\in\left[BC(\Gamma)\cap H^{1/2}_{\rm loc}(\Gamma)\right]^2$ and $\eta$ is a 
complex number satisfying ${\rm Re}(\eta)>0$. Throughout this paper, to ensure that 
${\bm \Pi}_{D,h}^{(2)}(x,y)$ has a weak singularity while $|x-y|\to 0$ for $x,y\in\Gamma$, 
$\widetilde{\mu}$ and $\widetilde{\lambda}$ in (\ref{a5}) are chosen to be
\be\label{xx1}
\widetilde{\mu}=\frac{\mu(\lambda+\mu)}{\lambda+3\mu},\quad\widetilde{\lambda}
=\frac{(\lambda+\mu)(\lambda+2\mu)}{\lambda+3\mu}.
\en
By doing so, it follows from Theorems 2.6 and 2.7, and Lemma 2.8 in \cite{TA02} that 
$\bm{u}^{\rm sc}$ given by (\ref{a9}) satisfies 
$\bm{u}^{\rm sc}\in\left[C^2(\Om)\cap C(\ov{\Omega})\cap H^1_{\rm loc}(\Omega)\right]^2$. 
By Theorems 2.4 and 3.2 in \cite{TA02}, it can be deduced that
$\bm{u}^{\rm sc}$ given by (\ref{a9}) satisfies the Navier equation (\ref{a1}) and
the UPRC (\ref{a3}). Further, as a consequence of Theorems 2.1 and 2.3 {\color{ljl} in \cite{TA02}}, 
$\bm{u}^{\rm sc}$ given by (\ref{a9}) satisfies the vertical growth rate condition (\ref{a8}) 
with $\beta=-1/2$.

According to the jump relations for elastic single- and double-layer potentials shown in
Theorems 2.6 and 2.7 in \cite{TA02}, it is easy to see that $\bm{u}^{\rm sc}$ given by (\ref{a9}) 
is a solution to  the problem (DP) provided $\bm{\psi}$ is a solution to the following integral equation
\begin{eqnarray}\label{a10}
\frac{1}{2}\bm{\psi}(x)+\int_{\Gamma}\left[\bm{\Pi}_{D,h}^{(2)}(x,y)-{\rm i}\eta\bm{G}_{D,h}(x,y)\right]\bm{\psi}(y)ds(y)=\bm{g}(x),\quad x\in\Gamma,
\end{eqnarray}
which can be rewritten in the operator form
\begin{eqnarray}\label{a11}
(\bm{I}+\bm{D}_\Gamma-{\rm i}\eta\bm{S}_\Gamma)\bm{\psi}=2\bm{g}\quad{\rm on}\;\;\Gamma,
\end{eqnarray}
where $\bm{D}_\Gamma$ and $\bm{S}_\Gamma$ are the elastic double-layer and single-layer operators given by
\begin{eqnarray*}\label{a12}
(\bm{D}_\Gamma\bm{\psi})(x):=2\int_{\Gamma}\bm{\Pi}_{D,h}^{(2)}(x,y)\bm{\psi}(y)ds(y),\quad x\in\Gamma,\\ \label{a13}
(\bm{S}_\Gamma\bm{\psi})(x):=2\int_{\Gamma}\bm{G}_{D,h}(x,y)\bm{\psi}(y)ds(y),\quad x\in\Gamma.
\end{eqnarray*}

For $x,y\in\Gamma$, we denote $x=x(s)=(s,f(s))$, $y=y(t)=(t,f(t))$, 
$\widetilde{\bm{\psi}}(t):=\bm{\psi}(y(t))$, and $\widetilde{\bm{g}}(s):=\bm{g}(x(s))$.
By changing the variables, we rewrite $\bm{D}_\Gamma$ and $\bm{S}_\Gamma$ as the following operators
\begin{eqnarray}\label{a14}
(\bm{D}\widetilde{\bm{\psi}})(s):=2\int_{\R}\bm{\Pi}_{D,h}^{(2)}(x(s),y(t))
\widetilde{\bm{\psi}}(t)\sqrt{1+f'(t)^2}dt,\quad s\in\R,\\\label{a15}
(\bm{S}\widetilde{\bm{\psi}})(s):=2\int_{\R}\bm{G}_{D,h}(x(s),y(t))
\widetilde{\bm{\psi}}(t)\sqrt{1+f'(t)^2}dt,\quad s\in\R.
\end{eqnarray}
Then the solvability of (\ref{a11}) in $\left[BC(\Gamma)\right]^2$ is equivalent to finding the solution $\widetilde{\bm{\psi}}$ to the integral equation 
\begin{eqnarray}\label{a16}
(\bm{I}+\bm{D}-{\rm i}\eta\bm{S})\widetilde{\bm{\psi}}=2\widetilde{\bm{g}}
\end{eqnarray}
in $\left[BC(\R)\right]^2$, which is given in the following theorem. 

\begin{theorem}(\cite[Corollary 5.12]{TA02})\label{thm2}
For any $f\in B^{(0)}_{c,M}$, the integral operator $\bm{I}+\bm{D}-{\rm i}\eta\bm{S}: \left[BC(\R)\right]^2\to\left[BC(\R)\right]^2$ is bijective (and so boundedly invertible) with
\begin{eqnarray*}\label{a17}
\sup_{f\in B^{(0)}_{c,M}}\|(\bm{I}+\bm{D}-{\rm i}\eta\bm{S})^{-1}\|_{\left[BC(\R)\right]^2
\to\left[BC(\R)\right]^2}<\infty.
\end{eqnarray*}
Thus, the integral equation (\ref{a10}) and (\ref{a16}) have exactly one solution for every 
$f\in B^{(0)}_{c,M}$ with
\begin{eqnarray*}
\|\bm{u}^{\rm sc}\|_{\infty,\Gamma}\leq C\|\bm{g}\|_{\infty,\Gamma}
\end{eqnarray*}
for some constants $C>0$ depending only on $B^{(0)}_{c,M}$ and $\omega$.
\end{theorem}

\begin{remark}
By Theorems \ref{thm1} and \ref{thm2}, the problem (DP) has a unique solution.
\end{remark}

\section{The singularity for the kernel of $\bm{D}-{\rm i}\eta \bm{S}$}\label{sec3}\setcounter{equation}{0}

This section is devoted to analyzing the singularity of the boundary integral equation (\ref{a16}), 
which will be used for the further investigation on the Nystr\"{o}m method in Section \ref{sec4}.
The main idea is to write the kernel $\bm{A}(s,t)$ of the integral operator 
$\bm{D}-{\rm i}\eta\bm{S}$ in (\ref{a16}) in the following form
\begin{eqnarray}\label{b1}
\bm{A}(s,t)=\frac{1}{2\pi}\bm{B}(s,t)\ln\left(4\sin^2\frac{s-t}{2}\right)+\bm{C}(s,t),\quad s,t\in\R,s\neq t,
\end{eqnarray}
with smooth matrix functions $\bm{B}(s,t)$ and $\bm{C}(s,t)$ (see the formulas (\ref{e7}) and (\ref{e8}) below). For the details on the smoothness of $\bm{B}(s,t)$ and $\bm{C}(s,t)$,
see Remark \ref{thm42}.

According to (\ref{a4}), (\ref{a6}), (\ref{a14}), and (\ref{a15}), the operators $\bm{D}$ and $\bm{S}$ can be decomposed into two parts as follows
\begin{eqnarray*}\label{b7}
\bm{D}=\bm{D}_1-\bm{D}_2\quad{\rm and}\quad\bm{S}=\bm{S}_1-\bm{S}_2,
\end{eqnarray*}
where
\begin{eqnarray}\label{b9}
&&(\bm{D}_1\widetilde{\bm{\psi})}(s):=2\int_{\Gamma}\bm{\Pi}_1^{(2)}(x(s),y(t))
\widetilde{\bm{\psi}}(t)\sqrt{1+f'(t)^2}dt,\quad s\in\R,\\ \label{b10}
&&(\bm{D}_2\widetilde{\bm{\psi})}(s):=2\int_{\Gamma}\bm{\Pi}_2^{(2)}(x(s),y(t))
\widetilde{\bm{\psi}}(t)\sqrt{1+f'(t)^2}dt,\quad s\in\R,\\\label{b11}
&&(\bm{S}_1\widetilde{\bm{\psi})}(s):=2\int_{\Gamma}\bm{G}(x(s),y(t))
\widetilde{\bm{\psi}}(t)\sqrt{1+f'(t)^2}dt,\quad s\in\R,\\\label{b12}
&&(\bm{S}_2\widetilde{\bm{\psi})}(s):=2\int_{\Gamma}\left[\bm{G}(x(s),y'(t))
-\bm{U}(x(s),y(t))\right]\widetilde{\bm{\psi}}(t)\sqrt{1+f'(t)^2}dt,\quad s\in\R,
\end{eqnarray}
with the components of $\bm{\Pi}_1^{(2)}(x,y)$ and $\bm{\Pi}_2^{(2)}(x,y)$ given by
\begin{eqnarray}\label{b13}
\Pi_{1,jk}^{(2)}(x,y):=\left(\bm{P}^{(y)}\left(\bm{G}_{j\cdot}(x,y)\right)^\top\right)_k
\quad{\rm and}\quad \Pi_{2,jk}^{(2)}(x,y):=\left(\bm{P}^{(y)}\left(\bm{G}_{j\cdot}(x,y')-\bm{U}_{j\cdot}(x,y)\right)^\top\right)_k.
\end{eqnarray}
Hence, the integral equation (\ref{a16}) can be rewritten as
\begin{eqnarray*}\label{b14}
\left[\bm{I}+\bm{D}_1-{\rm i}\eta \bm{S}_1-(\bm{D}_2-{\rm i}\eta \bm{S}_2)\right]\widetilde{\bm{\psi}}=2\widetilde{\bm{g}} \quad{\rm on}\;\;\R.
\end{eqnarray*}
The remaining part of this section consists of three subsections, which focus on the singularity 
analysis of the kernels in the integral operators $\bm{S}_1$, $\bm{D}_1$, and 
$\bm{D}_2-{\rm i}\eta\bm{S}_2$, respectively.

\subsection{Separating the logarithmic part of $\bm{S}_1$}

This subsection is devoted to separating the logarithmic part of the operator $\bm{S}_1$. 
We first introduce some notations which will be used later. For $x(s),~y(t)\in\Gamma$, 
we define the distance between $x(s)$ and $y(t)$ as
\begin{eqnarray*}
r=r(s,t):=|x(s)-y(t)|=\sqrt{(s-t)^2+(f(s)-f(t))^2},
\end{eqnarray*}
and define the upward unit normal at $x(s)$ and $y(t)$ as
\begin{eqnarray*}
&&\nu(s)=(\nu_1(s),\nu_2(s))^\top\quad{\rm with}\quad \nu_1(s)=-\frac{f'(s)}{\sqrt{1+f'(s)^2}}\;\;{\rm and}\;\;  \nu_2(s)=\frac{1}{\sqrt{1+f'(s)^2}},\\
&&\nu(t)=(\nu_1(t),\nu_2(t))^\top\quad\;{\rm with}\quad \nu_1(t)=-\frac{f'(t)}{\sqrt{1+f'(t)^2}}\;\;\;{\rm and}\;\;  \nu_2(t)=\frac{1}{\sqrt{1+f'(t)^2}}.
\end{eqnarray*}
{\color{ljl}Then for convenience, we define the vector $l$ and $l^{\perp}$ as }
\begin{eqnarray*}
&&l(s)=\sqrt{1+f'(s)^2}\nu(s)\quad{\rm and}\quad l^{\perp}(s)=\sqrt{1+f'(s)^2}\nu^{\perp}(s),\\
&&l(t)=\sqrt{1+f'(t)^2}\nu(t)\quad{\rm and}\quad l^{\perp}(t)=\sqrt{1+f'(t)^2}\nu^{\perp}(t).
\end{eqnarray*}

In terms of (\ref{a7}), a direct calculation shows that each element  $G_{jk}$ of the Green's tensor $\bm{G}$ can be represented as
\begin{eqnarray}\nonumber
G_{jk}(x,y)&=&\left[\frac{\rm i}{4\mu}H^{(1)}_0(\kappa_{\rm s} r)-\frac{\rm i}{4\om^2}\frac{\kappa_{\rm s}H^{(1)}_1(\kappa_{\rm s} r)-\kappa_{\rm p}H^{(1)}_1(\kappa_{\rm p} r)}{r}\right]\delta_{jk}\\
\label{c1}&&+\frac{\rm i}{4\omega^2}\frac{\kappa_{\rm s}^2H^{(1)}_2(\kappa_{\rm s} r)-\kappa_{\rm p}^2H^{(1)}_2(\kappa_{\rm p} r)}{r^2}(x_j-y_j)(x_k-y_k),\quad j,k=1,2,
\end{eqnarray}
where $x=(x_1,x_2),~y=(y_1,y_2)$ and $\delta_{jk}~(j,k=1,2)$ is the Kronecker delta function satisfying $\delta_{jk}=1$ for $j=k$ and $\delta_{jk}=0$ for $j\neq k$. Substituting $x=x(s):=\left(x_1(s),x_2(s)\right)$ and $y=y(t):=\left(y_1(t),y_2(t)\right)$ into (\ref{b11}) gives that $\bm{S}_1$ can be rewritten as
\begin{eqnarray*}\label{c2}
(\bm{S}_1\widetilde{\bm{\psi}})(s)=\int_{\R}\bm{A}^{(1)}(s,t)\widetilde{\bm{\psi}}(t)dt,\quad
s\in\R,
\end{eqnarray*}
with the element of the matrix $\bm{A}^{(1)}(s,t)$ given by
\begin{eqnarray*}\nonumber
A_{jk}^{(1)}(s,t)&=&2G_{jk}(x(s),y(t))\sqrt{1+f'(t)^2}\\\nonumber
&=&\left\{\left[\frac{\rm i}{2\mu}H_0^{(1)}(\kappa_{\rm s}r)-\frac{\rm i}{2\omega^2}\frac{\kappa_{\rm s}H_1^{(1)}(\kappa_{\rm s}r)-\kappa_{\rm p}H_1^{(1)}(\kappa_{\rm p}r)}{r}\right]\delta_{jk}\right.\\\label{c3}
&&\left.+\frac{\rm i}{2\omega^2}\frac{\kappa_{\rm s}^2H_2^{(1)}(\kappa_{\rm s}r)-\kappa_{\rm p}^2H_2^{(1)}(\kappa_{\rm p}r)}{r^2}(x_j(s)-y_j(t))(x_k(s)-y_k(t))\right\}\sqrt{1+f'(t)^2},
\end{eqnarray*}
for $j,k=1,2$. Based on the singularity of $A^{(1)}_{jk}(s,t)$, we can separate the logarithmic part and 
decompose $A_{jk}^{(1)}(s,t)$ as
\begin{eqnarray*}\label{c4}
A_{jk}^{(1)}(s,t)=B_{jk}^{(1)}(s,t)\ln |s-t|+C_{jk}^{(1)}(s,t),
\end{eqnarray*}
where
\begin{eqnarray}\nonumber
B_{jk}^{(1)}(s,t)&=&\frac{1}{\pi}\left\{\left[-\frac{1}{\mu}J_0(\kappa_{\rm s} r)+\frac{1}{\om^2}\frac{\kappa_{\rm s}J_1(\kappa_{\rm s} r)-\kappa_{\rm p}J_1(\kappa_{\rm p} r)}{r}\right]\delta_{jk} \right.\\\label{c5}
 &&-\left.\frac{1}{\om^2}\frac{\kappa_{\rm s}^2J_2(\kappa_{\rm s} r)-\kappa_{\rm p}^2J_2(\kappa_{\rm p} r)}{r^2}(x_j(s)-y_j(t))(x_k(s)-y_k(t))\right\}\sqrt{1+f'(t)^2},\\\label{c6}
C_{jk}^{(1)}(s,t)&=&A_{jk}^{(1)}(s,t)-B_{jk}^{(1)}(s,t)\ln |s-t|.
\end{eqnarray}
To get the exact expressions of $B_{jk}^{(1)}(s,t)$ and $C_{jk}^{(1)}(s,t)$ while $s=t$ 
for numerical computation, we need to use the following asymptotic behavior of Bessel functions 
(see \cite[(5.16.1)--(5.16.3)]{NN} and \cite{LL19}): as $r\to 0$,
\begin{eqnarray}\label{c7}
&&Y_0(r)\approx-\frac{2}{\pi}\ln\frac{2}{r},\quad H_0^{(1)}(r)\approx 
-\frac{2{\rm i}}{\pi}\ln\frac{2}{r},\\\label{c8}
&&J_n(r)\approx\frac{r^n}{2^n\Gamma(1+n)},\quad Y_n(r)\approx-\frac{\Gamma(n)}{\pi}\left(\frac{2}{r}\right)^n, 
\quad H_n^{(1)}(r)\approx -\frac{{\rm i}\Gamma(n)}{\pi}\left(\frac{2}{r}\right)^n,\;\; n> 0,\\ \label{d6}
&&\kappa_{\rm s}^2H_2^{(1)}(\kappa_{\rm s}r)-\kappa_{\rm p}^2H_2^{(1)}(\kappa_{\rm p}r)
=\frac{\rm i}{\pi}(\kappa_{\rm p}^2-\kappa_{\rm s}^2)+\frac{\rm i}{4\pi}\left(\kappa_{\rm s}^4
\ln\frac{\kappa_{\rm s}r}{2}-\kappa_{\rm p}^4\ln\frac{\kappa_{\rm p}r}{2}\right)r^2+O(r^2),\\ \label{d7}
&&\kappa_{\rm s}^3H_3^{(1)}(\kappa_{\rm s}r)-\kappa_{\rm p}^3H_3^{(1)}(\kappa_{\rm p}r)
=\frac{2\rm i}{\pi}(\kappa_{\rm p}^2-\kappa_{\rm s}^2)\frac{1}{r}+\frac{\rm i}{4\pi}(\kappa_{\rm p}^4
-\kappa_{\rm s}^4)r+O(r^3\ln\frac{r}{2}),
\end{eqnarray}
where $\Gamma(n)$ denotes the gamma function, and the notation $\varphi_1(r)\approx\varphi_2(r)$ 
for functions $\varphi_1$ and $\varphi_2$ means that $\lim\limits_{r\to 0}\varphi_1(r)/\varphi_2(r)=1$. 
With the aid of $H_n^{(1)}(r)=J_n(r)+{\rm i}Y_n(r)$ and the ascending series expansions of the Bessel 
functions (see \cite[(3.97), (3.98)]{CK19}), we can derive the following limits and asymptotic formulas
\begin{align*}
\lim_{r\to 0}\left[H_0^{(1)}(\kappa_{\rm s}r)-\frac{2{\rm i}}{\pi}J_0(\kappa_{\rm s}r)\ln(\kappa_{\rm s}r)\right]
&=\alpha,\\
\lim_{r\to 0}\frac{1}{r}\left\{\kappa_{\rm s}\left[H_1^{(1)}(\kappa_{\rm s}r)-\frac{2{\rm i}}{\pi}J_1(\kappa_{\rm s}r)\ln(\kappa_{\rm s}r)\right]-\kappa_{\rm p}\left[H_1^{(1)}(\kappa_{\rm p}r)-\frac{2{\rm i}}{\pi}J_1(\kappa_{\rm p}r)\ln(\kappa_{\rm p}r)\right]\right\}
&=\frac{\alpha}{2}(\kappa_{\rm s}^2-\kappa_{\rm p}^2),\\\nonumber
\frac{1}{r^2}\left\{\kappa_{\rm s}^2\left[H_2^{(1)}(\kappa_{\rm s}r)-\frac{2{\rm i}}{\pi}J_2(\kappa_{\rm s}r)\ln(\kappa_{\rm s}r)\right]-\kappa_{\rm p}^2\left[H_2^{(1)}(\kappa_{\rm p}r)-\frac{2{\rm i}}{\pi}J_2(\kappa_{\rm p}r)\ln(\kappa_{\rm p}r)\right]\right\}
&\approx \frac{\alpha}{8}(\kappa_{\rm s}^4-\kappa_{\rm p}^4)-\frac{\rm i}{\pi r^2}(\kappa_{\rm s}^2-\kappa_{\rm p}^2),
\end{align*}
where $\alpha=1+\frac{2{\rm i}}{\pi}(C_E-\ln 2)$ with $C_E$ standing for the Euler's constant . The above asymptotic formulas implies that the diagonal terms are
\begin{eqnarray}\label{c12}
B_{jk}^{(1)}(s,s)&=&\frac{1}{\pi}\left[-\frac{1}{\mu}+1/2(c_{\rm s}^2-c_{\rm p}^2)\right]\delta_{jk}\sqrt{1+f'(s)^2},\\\nonumber
C_{jk}^{(1)}(s,s)&=&\sqrt{1+f'(s)^2}\Bigg\{\left[\frac{\rm i}{4}\alpha (c_{\rm s}^2+c_{\rm p}^2)-\frac{1}{2\pi}\left(c_{\rm s}^2\ln(\kappa_{\rm s}\sqrt{1+f'(s)^2})+c_{\rm p}^2\ln(\kappa_{\rm p}\sqrt{1+f'(s)^2})\right)\right]\delta_{jk}\\\label{c13}
&&+\frac{1}{2\pi}(c_{\rm s}^2-c_{\rm p}^2)\frac{l^{\perp}_j(s) l^{\perp}_k(s)}{1+f'(s)^2}\Bigg\},
\end{eqnarray}
where $l^{\perp}_j$ and $l^{\perp}_k~(j,k=1,2)$ denote the components of the vector $l^{\perp}$.

\subsection{Separating the logarithmic part of $\bm{D}_1$}

The purpose of this subsection is to separate the logarithmic part of the operator $\bm{D}_1$. 
Due to (\ref{b9}), (\ref{b13}), (\ref{a5}), and (\ref{c1}), we can write $\bm{D}_1$ as
\begin{eqnarray*}\label{d1}
(\bm{D}_1\widetilde{\bm{\psi}})(s)=\int_{\R}\bm{A}^{(2)}(s,t)\widetilde{\bm{\psi}}(t)dt,
\quad s\in\R,
\end{eqnarray*}
where {\color{ljl} the elements of the matrix $\bm{A}^{(2)}(s,t)$ are given by}
\begin{eqnarray*}\nonumber
A_{jk}^{(2)}(s,t)&:=&2\left[\bm{P}^{(y)}\left(\bm{G}_{j\cdot}(x,y)\right)^\top\right]_k\sqrt{1+f'(t)^2}\\\nonumber
&=&(\mu+\widetilde{\mu})\left\{\left[-\frac{\rm i}{2\mu}\frac{\kappa_{\rm s}H_1^{(1)}(\kappa_{\rm s}r)}{r}+\frac{\rm i}{2\omega^2}\frac{\kappa_{\rm s}^2H_2^{(1)}(\kappa_{\rm s}r)-\kappa_{\rm p}^2H_2^{(1)}(\kappa_{\rm p}r)}{r^2}\right]\delta_{jk}\right.\\\nonumber
&&\left.-\frac{\rm i}{2\omega^2}\frac{\kappa_{\rm s}^3H_3^{(1)}(\kappa_{\rm s}r)-\kappa_{\rm p}^3H_3^{(1)}(\kappa_{\rm p}r)}{r^3}(x_j(s)-y_j(t))(x_k(s)-y_k(t))\right\}\left[(s-t)f'(t)+f(t)-f(s)\right]\\\nonumber
&&-(\mu+\widetilde{\mu})\frac{\rm i}{2\omega^2}\frac{\kappa_{\rm s}^2H_2^{(1)}(\kappa_{\rm s}r)-\kappa_{\rm p}^2H_2^{(1)}(\kappa_{\rm p}r)}{r^2}\left[(x_j(s)-y_j(t))l_k(t)+(x_k(s)-y_k(t))l_j(t)\right]\\\nonumber
&&+\widetilde{\lambda}\left[\frac{\rm i}{2\mu}\frac{\kappa_{\rm s}H_1^{(1)}(\kappa_{\rm s}r)}{r}-\frac{2\rm i}{\omega^2}\frac{\kappa_{\rm s}^2H_2^{(1)}(\kappa_{\rm s}r)-\kappa_{\rm p}^2H_2^{(1)}(\kappa_{\rm p}r)}{r^2}\right.\\\nonumber
&&\left.+\frac{\rm i}{2\omega^2}\frac{\kappa_{\rm s}^3H_3^{(1)}(\kappa_{\rm s}r)-\kappa_{\rm p}^3H_3^{(1)}(\kappa_{\rm p}r)}{r}\right](x_j(s)-y_j(t))l_k(t)\\\label{d2}
&&-\widetilde{\mu}\frac{\rm i}{2\mu}\frac{\kappa_{\rm s}H_1^{(1)}(\kappa_{\rm s}r)}{r}\left[(f(s)-f(t))\delta_{j1}-(s-t)\delta_{j2}\right]l^{\perp}_k(t).
\end{eqnarray*}
{\color{ljl}Based} on the singularity of $A_{jk}^{(2)}(s,t)$, we can separate the logarithmic part of $A_{jk}^{(2)}(s,t)$ as
\begin{eqnarray*}\label{d3}
A_{jk}^{(2)}(s,t)=B_{jk}^{(2)}(s,t)\ln |s-t|+C_{jk}^{(2)}(s,t),
\end{eqnarray*}
where
\begin{eqnarray}\nonumber
B_{jk}^{(2)}(s,t)&:=&(\mu+\widetilde{\mu})\left\{\left[\frac{1}{\pi\mu}\frac{\kappa_{\rm s}J_1(\kappa_{\rm s}r)}{r}-\frac{1}{\pi\omega^2}\frac{\kappa_{\rm s}^2J_2(\kappa_{\rm s}r)-\kappa_{\rm p}^2J_2(\kappa_{\rm p}r)}{r^2}\right]\delta_{jk}\right.\\\nonumber
&&+\left.\frac{1}{\pi\omega^2}\frac{\kappa_{\rm s}^3J_3(\kappa_{\rm s}r)-\kappa_{\rm p}^3J_3(\kappa_{\rm p}r)}{r^3}(x_j(s)-y_j(t))(x_k(s)-y_k(t))\right\}\left[(s-t)f'(t)+f(t)-f(s)\right]\\\nonumber
&&+(\mu+\widetilde{\mu})\frac{1}{\pi\omega^2}\frac{\kappa_{\rm s}^2J_2(\kappa_{\rm s}r)-\kappa_{\rm p}^2J_2(\kappa_{\rm p}r)}{r^2}\left[(x_j(s)-y_j(t))l_k(t)+(x_k(s)-y_k(t))l_j(t)\right]\\\nonumber
&&+\widetilde{\lambda}\left[-\frac{1}{\pi\mu}\frac{\kappa_{\rm s}J_1(\kappa_{\rm s}r)}{r}+\frac{4}{\pi\omega^2}\frac{\kappa_{\rm s}^2J_2(\kappa_{\rm s}r)
	-\kappa_{\rm p}^2J_2(\kappa_{\rm p}r)}{r^2}\right.\\\nonumber
&&-\left.\frac{1}{\pi\omega^2}\frac{\kappa_{\rm s}^3J_3(\kappa_{\rm s}r)-\kappa_{\rm p}^3J_3(\kappa_{\rm p}r)}{r}\right](x_j(s)-y_j(t))l_k(t)\\\label{d4}
&&+\widetilde{\mu}\frac{1}{\pi\mu}\frac{\kappa_{\rm s}J_1(\kappa_{\rm s}r)}{r}\left[(f(s)-f(t))\delta_{j1}-(s-t)\delta_{j2}\right]l^{\perp}_k(t),\label{d5}\\
C_{jk}^{(2)}(s,t)&:=&A_{jk}^{(2)}(s,t)-B_{jk}^{(2)}(s,t)\ln |s-t|.\label{d6}
\end{eqnarray}
Similar as the previous subsection, we need to get the exact expressions of $B_{jk}^{(2)}(s,t)$ and $C_{jk}^{(2)}(s,t)$ while $s=t$ for numerical computation.
Using equations (\ref{c7})--(\ref{d7}), choosing $\widetilde{\mu}$ as in (\ref{xx1}) and by a direct 
but lengthy calculation we obtain the diagonal terms
\begin{eqnarray}\label{d8}
&&B_{jk}^{(2)}(s,s)=0,\\\label{d9}
&&C_{jk}^{(2)}(s,s)=-\frac{1}{2\pi}\frac{f''(s)}{1+f'(s)^2}\left\{\left[-1+\frac{\mu+\widetilde{\mu}}{2}
(c_{\rm s}^2-c_{\rm p}^2)\right]\delta_{jk}
+(c_{\rm p}^2-c_{\rm s}^2)\frac{l_j^{\perp}(s)l_k^{\perp}(s)}{1+f'(s)^2}\right\}.
\end{eqnarray}

\subsection{The computation of $\bm{D}_2-{\rm i}\eta\bm{S}_2$}

Observing from (\ref{b10}), (\ref{b12}), and (\ref{b13}), the kernel of the integral operator 
$\bm{D}_2-{\rm i}\eta\bm{S}_2$ is related to $\bm{G}(x,y')$ and $\bm{U}(x,y)$ with $x=(s, f(s))\in\Gamma$,  
$y=(t, f(t))\in\Gamma$, and $y'=(t, 2h-f(t))$. Due to $h<\inf_{x_1\in\R}f(x_1)$, it is readily seen that 
there is a positive distance between $x$ and $y'$, which leads to that $\bm{G}(x,y')$ is smooth, and 
by \cite[Theorem 2.1]{TA02}, we have 
$\bm{U}(x,y)\in \left[C^{\infty}(U_h)\cap C^1(\overline{U_h})\right]^{2\times 2}$. 
It follows from (\ref{b10}), (\ref{b12}), (\ref{b13}), (\ref{a5}), (\ref{a7}) and a direct but lengthy 
calculation that the integral operator $\bm{D}_2-{\rm i}\eta\bm{S}_2$ can be rewritten as integral 
on the real line, which reads
\begin{eqnarray*}\label{e1}
\left[(\bm{D}_2-{\rm i}\eta\bm{S}_2)\widetilde{\bm{\psi}}\right](s)
=\int_{\R}\bm{A}^{(3)}(s,t)\widetilde{\bm{\psi}}(t)dt,\quad s\in\R,
\end{eqnarray*}
where
\begin{eqnarray}\label{x2}
 \bm{A}^{(3)}(s,t)=\bm{A}^{(4)}(s,t)-\bm{A}^{(5)}(s,t)
 \end{eqnarray}
with
\begin{eqnarray}\nonumber
&&A^{(4)}_{jk}(s,t):=2\left\{\left[\bm{P}^{(y)}\left(\bm{G}_{j\cdot}(x,y')\right)^\top\right]_k-{\rm i}
\eta G_{jk}(x,y')\right\}\sqrt{1+f'(t)^2}=:I_1+I_2+I_3+I_4,\\\nonumber
&&A_{jk}^{(5)}(s,t):=2\left\{\left[\bm{P}^{(y)}\left(\bm{U}_{j\cdot}(x,y')\right)^\top\right]_k-{\rm i}
\eta U_{jk}(x,y')\right\}\sqrt{1+f'(t)^2}.
\end{eqnarray}
Here, $I_j$ $(j=1,2,3,4)$ are defined by
\begin{eqnarray*}
I_1&=&(\mu+\widetilde{\mu})\left\{\left[-\frac{\rm i}{2\mu}\frac{\kappa_{\rm s}H_1^{(1)}(\kappa_{\rm s}r')}{r'}\delta_{jk}+\frac{\rm i}{2\omega^2}\frac{\kappa_{\rm s}^2H_2^{(1)}(\kappa_{\rm s}r')-\kappa_{\rm p}^2H_2^{(1)}(\kappa_{\rm p}r')}{r'^2}\delta_{jk}\right.\right.\\\nonumber
&&\left.\left.-\frac{\rm i}{2\omega^2}\frac{\kappa_{\rm s}^3H_3^{(1)}(\kappa_{\rm s}r')-\kappa_{\rm p}^3H_3^{(1)}(\kappa_{\rm p}r')}{r'^3}(x_j(s)-y'_j(t))(x_k(s)-y'_k(t))\right]\left[(s-t)f'(t)+f(t)+f(s)-2h\right]\right.\\\nonumber
&&\left.+\frac{\rm i}{2\omega^2}\frac{\kappa_{\rm s}^2H_2^{(1)}(\kappa_{\rm s}r')-\kappa_{\rm p}^2H_2^{(1)}(\kappa_{\rm p}r')}{r'^2}\left[(x_k(s)-y'_k(t))q_jl_j(t)+(x_j(s)-y'_j(t))q_kl_k(t)\right]\right\},
\\
I_2&=&\widetilde{\lambda}\left\{\left[-\frac{\rm i}{2\mu}\frac{\kappa_{\rm s}H_1^{(1)}(\kappa_{\rm s}r')}{r'}+\frac{\rm i}{\omega^2}\frac{\kappa_{\rm s}^2H_2^{(1)}(\kappa_{\rm s}r')-\kappa_{\rm p}^2H_2^{(1)}(\kappa_{\rm p}r')}{r'^2}\right](y_j(s)-x'_j(t))\right.\\\nonumber
&&\left.+\frac{\rm i}{2\omega^2}\frac{\kappa_{\rm s}^3H_3^{(1)}(\kappa_{\rm s}r')-\kappa_{\rm p}^3H_3^{(1)}(\kappa_{\rm p}r')}{r'^3}(x_j(s)-y'_j(t))\left[(t-s)^2-(f(t)+f(s)-2h)^2\right]\right\}l_k(t),
\\
I_3&=&-\widetilde{\mu}\left\{\left[-\frac{\rm i}{2\mu}\frac{\kappa_{\rm s}H_1^{(1)}(\kappa_{\rm s}r')}{r'}+\frac{\rm i}{\omega^2}\frac{\kappa_{\rm s}^2H_2^{(1)}(\kappa_{\rm s}r')-\kappa_{\rm p}^2H_2^{(1)}(\kappa_{\rm p}r')}{r'^2}\right]\left[(s-t)\delta_{j2}+(f(s)+f(t)-2h)\delta_{j1}\right]\right.\\\nonumber
&&\left.-\frac{\rm i}{\omega^2}\frac{\kappa_{\rm s}^3H_3^{(1)}(\kappa_{\rm s}r')-\kappa_{\rm p}^3H_3^{(1)}(\kappa_{\rm p}r')}{r'^3}(s-t)(f(s)+f(t)-2h)(x_j(s)-y'_j(t))\right\}l_k^{\perp}(t),
\\
I_4&=&\eta\left\{\left[\frac{1}{2\mu}H_0^{(1)}(\kappa_{\rm s}r')-\frac{1}{2\omega^2}\frac{\kappa_{\rm s}H_1^{(1)}(\kappa_{\rm s}r')-\kappa_{\rm p}H_1^{(1)}(\kappa_{\rm p}r')}{r'}\right]\delta_{jk}\right.\\
&&\left.+\frac{1}{2\omega^2}\frac{\kappa_{\rm s}^2H_2^{(1)}(\kappa_{\rm s}r')-\kappa_{\rm p}^2H_2^{(1)}(\kappa_{\rm p}r')}{r'^2}(x_j(s)-y'_j(t))(x_k(s)-y'_k(t))\right\}\sqrt{1+f'(t)^2},
\end{eqnarray*}
where $q=(q_1, q_2)^\top=(-1,1)^\top$ and $r'$ denotes the distance between $x$ and $y'$, that is,
\begin{eqnarray*}
r'=r'(s,t):=\sqrt{(s-t)^2+(f(s)+f(t)-2h)^2}.
\end{eqnarray*}

According to {\color{ljl}the above three subsections on the analysis of $\bm{S}_1$, $\bm{D}_1$, 
and $\bm{D}_2-{\rm i}\eta\bm{S}_2$}, the integral operator $\bm{D}-{\rm i}\eta\bm{S}$ can be 
rewritten in the following form
\begin{eqnarray}\label{e3}
\left[(\bm{D}-{\rm i}\eta\bm{S})\widetilde{\bm{\psi}}\right](s)=\int_{\R}\bm{A}(s,t)
\widetilde{\bm{\psi}}(t)dt
\end{eqnarray}
with
\begin{eqnarray*}\label{e4}
\bm{A}(s,t)=-{\rm i}\eta \bm{A}^{(1)}(s,t)+\bm{A}^{(2)}(s,t)-\bm{A}^{(3)}(s,t)
:=\bm{B}^*(s,t)\ln |s-t|+\bm{C}^*(s,t),
\end{eqnarray*}
where $\bm{B}^*(s,t)$ and $\bm{C}^*(s,t)$ are
\begin{eqnarray}\label{e5}
\bm{B}^*(s,t)&:=&-{\rm i}\eta \bm{B}^{(1)}(s,t)+\bm{B}^{(2)}(s,t),\\\label{e6}
\bm{C}^*(s,t)&:=&-{\rm i}\eta \bm{C}^{(1)}(s,t)+\bm{C}^{(2)}(s,t)-\bm{A}^{(3)}(s,t).
\end{eqnarray}
In order to employ the Nystr\"{o}m method, we follow the ideas of \cite[Theorem 2.1]{MACK} and 
rewrite the integral kernel $\bm{A}(s,t)$ in the form (\ref{b1}) with $\bm{B}(s,t)$ and $\bm{C}(s,t)$ given by
\begin{eqnarray}\label{e7}
&&\bm{B}(s,t):=\pi \bm{B}^*(s,t)\chi(s-t),\\\label{e8}
&&\bm{C}(s,t):=\bm{B}^*(s,t)\left[(1-\chi(s-t))\ln|s-t|-\chi(s-t)
\ln\left(\frac{\sin(\frac{s-t}{2}}{\frac{s-t}{2}}\right)\right]
+\bm{C}^*(s,t),
\end{eqnarray}
for $s\neq t$, where $\chi\in C_0^{\infty}(\R)$ is a cut-off function defined by
\begin{equation}\no
	\chi(s)=\left\{\begin{array}{cc}1, &|s|\leq1,\\\left[1+\exp\left(\frac{1}{\pi-|s|}+\frac{1}{1-|s|}\right)\right]^{-1}, &
		1<|s|<\pi,\\0, &\pi\leq|s|.\end{array}\right.
\end{equation}
It is easy to see that $\chi$ satisfies $\chi(s)\in [0,1]\;{\rm for}\; s\in \R$, 
$\chi(s)=0\;{\rm for}\;|s|\geq \pi$, $ \chi(s)=1\;{\rm for}\;|s|\leq 1$, and 
$\chi(-s)=\chi(s)\;{\rm for}\;s\in \R.$

Finally, with the help of (\ref{c12}), (\ref{c13}), (\ref{d8}), (\ref{d9}), (\ref{e7}) 
and (\ref{e8}), we obtain that
\begin{eqnarray*}\label{e9}
\bm{B}(s,s)&:=&\pi\left[-{\rm i}\eta \bm{B}^{(1)}(s,s)+\bm{B}^{(2)}(s,s)\right],\\\label{e10}
\bm{C}(s,s)&:=&-{\rm i}\eta \bm{C}^{(1)}(s,s)+\bm{C}^{(2)}(s,s)-\bm{A}^{(3)}(s,s).
\end{eqnarray*}

\section{Convergence analysis of the Nystr\"{o}m method}\label{sec4}\setcounter{equation}{0}

The goal of this section is to establish the convergence result of the Nystr\"{o}m method for 
the boundary integral equation \eqref{a16}.
In views of \eqref{e3}, \eqref{a16} can be rewritten in the following form
\begin{eqnarray}\label{f2}
\widetilde{\bm{\psi}}(s)+\frac{1}{2\pi}\int_{-\infty}^{+\infty}\ln\left(4\sin^2\frac{s-t}{2}\right)
\bm{B}(s,t)\widetilde{\bm{\psi}}(t)dt+\int_{-\infty}^{+\infty}\bm{C}(s,t)\widetilde{\bm{\psi}}(t)dt
=2\widetilde{\bm{g}}(s),\quad s\in\R.
\end{eqnarray}
To get the numerical solution of (\ref{f2}), we truncate the infinite interval $(-\infty, +\infty)$ 
into a finite interval $(-cut, cut)$, and choose an equidistant set of knots $t_j:=-cut+j\pi/N$ 
for $j=0,1,...,2N cut/\pi$. If $\bm{B}(s,t)\in C_{0,\pi}^n(\R^2)$ and $\bm{C}(s,t)\in BC_p^n(\R^2)$ 
for some $p>1$ and some positive integer $n$, it follows from \cite{MACK} that the two integrals 
in (\ref{f2}) can be approximated by
\begin{eqnarray*}\label{f3} \frac{1}{2\pi}\int_{-\infty}^{+\infty}\ln\left(4\sin^2\frac{s-t}{2}\right)
\bm{B}(s,t)\widetilde{\bm{\psi}}(t)dt&\approx& \sum_{j\in\mathbb Z}R_j^{(N)}(s)\bm{B}(s,t_j)
\widetilde{\bm{\psi}}(t_j),\quad s\in\R,\\\label{f4}
\int_{-\infty}^{+\infty}\bm{C}(s,t)\widetilde{\bm{\psi}}(t)dt&\approx& 
\frac{\pi}{N}\sum_{j\in\mathbb Z}\bm{C}(s,t_j)\widetilde{\bm{\psi}}(t_j),\quad s\in\R,
\end{eqnarray*}
with the quadrature weights given by
\begin{eqnarray*}\label{f5}
R_j^{(N)}(s):=-\frac{1}{N}\left[\sum_{m=1}^{N-1}\frac{1}{m}\cos m(s-t_j)+\frac{1}{2N}\cos N(s-t_j)\right].
\end{eqnarray*}
{\color{ljl}Therefore, an approximated form of (\ref{f2}) is}
\begin{eqnarray}\label{f6}
\widetilde{\bm{\psi}}_N(s)+\sum_{j\in\mathbb Z}\alpha_j^{(N)}(s)\widetilde{\bm{\psi}}_N(t_j)=2\widetilde{\bm{g}}(s),\quad s\in\R,
\end{eqnarray}
with
\begin{eqnarray*}\label{f7}
\alpha_j^{(N)}(s):=R_j^{(N)}(s)\bm{B}(s,t_j)+\frac{\pi}{N}\bm{C}(s,t_j).
\end{eqnarray*}
{\color{ljl}The remaining part of this section is to study the convergence result for $\|\widetilde{\bm{\psi}}-\widetilde{\bm{\psi}}_N\|_{[L^{\infty}(\R)]^2}$, which is presented in the following theorem.}

\begin{theorem}\label{thm41}
Let $f\in B_{c,M}^{(n)}$ and $\widetilde{\bm{g}}\in [BC^{n}(\R)]^2$ for some $n\in\mathbb{N}$ and $c, M>0$. There {\color{ljl}exists $N_0\in\mathbb N$} such that (\ref{f6}) admits a uniquely determined numerical solution $\widetilde{\bm{\psi}}_N$ and
\begin{eqnarray*}\label{f8}
\|\widetilde{\bm{\psi}}-\widetilde{\bm{\psi}}_N\|_{[L^{\infty}(\R)]^2}\lesssim N^{-n}\|\widetilde{\bm{g}}\|_{[BC^n(\R)]^2}
\end{eqnarray*}
for $N>N_0$, where $\widetilde{\bm{\psi}}$ is the unique solution of (\ref{a16}).
\end{theorem}
\begin{proof}	
According to Theorem \ref{thm2}, the integral equation (\ref{a16}) has exactly one solution $\widetilde{\bm{\psi}}\in [BC(\R)]^2$ for every $\widetilde{\bm{g}}\in [BC(\R)]^2$ and there exists {\color{ljl}$C_0>0$ such that $\|(\bm{I}+\bm{D}-{\rm i}\eta \bm{S})^{-1}\|\leq C_0$}.
Then by \cite[Theorem 3.13]{MACK}, the statement of this theorem holds if $\bm{B}(s,t)\in C_{0,\pi}^n(\R^2)$ and $\bm{C}(s,t)\in BC_p^n(\R^2)$ for some $p>1$. With the help of  \cite[Theorem 2.1]{MACK}, it is equivalent to showing the following three conditions: for all $j,l\in\N$ with $j+l\leq n$, there exists constants $C>0$ and $p>1$ such that
\begin{eqnarray*}
&&{\bm C1}.\;\; \left|\frac{\partial^{j+l}\bm{B^*}(s,t)}{\partial s^j\partial t^l}\right|\leq C,\qquad\qquad\qquad\quad{s,t\in\R},\;\;|s-t|\leq \pi,\\
&&{\bm C2}.\;\; \left|\frac{\partial^{j+l}\bm{C^*}(s,t)}{\partial s^j\partial t^l}\right|\leq C,\qquad\qquad\qquad\quad{s,t\in\R},\;\;|s-t|\leq \pi,\\
&&{\bm C3}.\;\; \left|\frac{\partial^{j+l}\bm{A}(s,t)}{\partial s^j\partial t^l}\right|\leq C(1+|s-t|)^{-p},\qquad\;{s,t\in\R},\;\;|s-t|\geq \pi,
\end{eqnarray*}
where $\bm{B^*}(s,t)$ and $\bm{C^*}(s,t)$ are defined by (\ref{e5}) and (\ref{e6}), respectively. Thus, it suffices to show that ${\bm C1}$--${\bm C3}$ hold.

For the condition ${\bm C1}$, we recall that $\bm{B^*}(s,t)=-{\rm i}\eta \bm{B}^{(1)}(s,t)+\bm{B}^{(2)}(s,t)$, where $\bm{B}^{(1)}(s,t)$ and $\bm{B}^{(2)}(s,t)$ are defined by (\ref{c5}) and (\ref{d4}), respectively. By the asymptotic formulas (\ref{c8}) and the fact that $J_n(r)$ is analytic for all $r\in\R$, we obtain that $J_n(r)/r^n$ is also analytic for all $r\in\R$. Since $f\in B_{c,M}^{(n)}$, it follows from (\ref{c5}) and (\ref{d4}) that $\bm{B}^{(1)}(s,t)\in BC^n(\R^2)$ and $\bm{B}^{(2)}(s,t)\in BC^n(\R^2)$, which implies that ${\bm C1}$ holds.

For the condition ${\bm C2}$, we recall that $\bm{C^*}(s,t)=-{\rm i}\eta \bm{C}^{(1)}(s,t)+\bm{C}^{(2)}(s,t)-\bm{A}^{(3)}(s,t)$, where $\bm{C}^{(1)}(s,t)$, $\bm{C}^{(2)}(s,t)$, 
and $\bm{A}^{(3)}(s,t)$ are defined by (\ref{c6}), (\ref{d5}), and (\ref{x2}), respectively. 
To prove this condition, we first introduce the following notations
\begin{eqnarray*}\label{f9}
&&\rho_n(\kappa,s,t):=H_n^{(1)}(\kappa r)-\frac{2 {\rm i}}{\pi}J_n(\kappa r)
\ln |s-t|,\quad\gamma_1(s,t):=\frac{\kappa_{\rm s}\rho_1(\kappa_{\rm s}, s, t)
-\kappa_{\rm p}\rho_1(\kappa_{\rm p}, s, t)}{r},\\ \label{f11}
&&\gamma_2(s,t):=\kappa_{\rm s}^2\rho_2(\kappa_{\rm s},s,t)
-\kappa_{\rm p}^2\rho_2(\kappa_{\rm p}, s,t),\qquad\;\;
\gamma_3(s,t):=r\left[\kappa_{\rm s}^3\rho_3(\kappa_{\rm s}, s,t)
-\kappa_{\rm p}^3\rho_3(\kappa_{\rm p}, s,t)\right],\\ \label{f13}
&&\xi(s,t):=\frac{(s-t)f'(t)+f(t)-f(s)}{r^2},
\end{eqnarray*}
and for $j,k=1,2$,
\begin{eqnarray*}
&&\zeta_{jk}(s,t):=\frac{(x_j(s)-y_j(t))(x_k(s)-y_k(t))}{r^2},\\\nonumber
&&\sigma_{jk}(s,t):=-(\mu+\widetilde{\mu})\frac{\rm i}{2\omega^2}\frac{\gamma_2(s,t)}{r^2}\left[(x_j(s)-y_j(t))l_k(t)+(x_k(s)-y_k(t))l_j(t)\right]\\\nonumber
&&\qquad\qquad\quad+\widetilde{\lambda}\left[\frac{\rm i}{2\mu}\frac{\kappa_{\rm s}}{r}\rho_1(\kappa_{\rm s}, s, t)-\frac{2\rm i}{\omega^2}\frac{\gamma_2(s,t)}{r^2}+\frac{\rm i}{2\omega^2}\frac{\gamma_3(s,t)}{r^2}\right](x_j(s)-y_j(t))l_k(t)\\\label{f14}
&&\qquad\qquad\quad-\widetilde{\mu}\frac{\rm i}{2\mu}\frac{\kappa_{s}}{r}\rho_1(\kappa_{\rm s},s,t)\left[(f(s)-f(t))\delta_{j1}-(s-t)\delta_{j2}\right]l^{\perp}_k(t).
\end{eqnarray*}
Using these notations and (\ref{e6}), (\ref{c6}), and (\ref{d6}), we can rewrite the elements of $\bm{C^*}(s,t)$ as
\begin{eqnarray}\nonumber
C^*_{jk}(s,t)&=&-{\rm i}\eta\left\{\left[\frac{\rm i}{2\mu}\rho_0(\kappa_{\rm s}, s, t)-\frac{\rm i}{2\omega^2}\gamma_1(s,t)\right]\delta_{jk}+\frac{\rm i}{2\omega^2}\gamma_2(s,t)\zeta_{jk}(s,t)\right\}\sqrt{1+f'(t)^2}\\\nonumber
&&+(\mu+\widetilde{\mu})\left\{\left[-\frac{\rm i}{2\mu}\kappa_{\rm s}r\rho_1(\kappa_{\rm s}, s, t)+\frac{\rm i}{2\omega^2}\gamma_2(s,t)\right]\delta_{jk}-\frac{\rm i}{2\omega^2}\gamma_3(s,t)\zeta_{jk}(s,t)\right\}\xi(s,t)\\\label{f15}
&&+\sigma_{jk}(s,t)-A_{jk}^{(3)}(s,t), \quad\quad j,k=1,2.
\end{eqnarray}
With the help of $H_n^{(1)}(z)=J_n(z)+{\rm i}Y_n(z)$ and the ascending series expansions of the Bessel functions (see \cite[(3.97) and (3.98)]{CK19}), we obtain that
\begin{eqnarray}\label{f16}
&&\rho_0(\kappa, s,t)=\left(1+\frac{2\rm i}{\pi}C_E+\frac{2\rm i}{\pi}
\ln\frac{\kappa r}{2|s-t|}\right)J_0(\kappa r)-\frac{2\rm i}{\pi}
\sum_{p=0}^{+\infty}\frac{(-1)^p}{(p!)^2}\left(\frac{\kappa r}{2}\right)^{2p}\phi(p),\\ \nonumber
&&\kappa_{\rm s}r\rho_1(\kappa_{\rm s}, s, t)=\left(1+\frac{2\rm i}{\pi}C_E
+\frac{2\rm i}{\pi}\ln \frac{\kappa_{\rm s} r}{2|s-t|}\right)\kappa_{\rm s}rJ_1(\kappa_{\rm s} r)
-\frac{2{\rm i}}{\pi}\\ \label{f17}
&&\qquad\qquad\qquad-\frac{2\rm i}{\pi}\sum_{p=0}^{+\infty}\frac{(-1)^p}{p!(p+1)!}
\left(\frac{\kappa_{\rm s} r}{2}\right)^{2+2p}\left(\phi(p+1)+\phi(p)\right),\\ \nonumber
&&\gamma_1(s,t)=\kappa_{\rm s}\left(1+\frac{2\rm i}{\pi}C_E+\frac{2\rm i}{\pi}
\ln\frac{\kappa_{\rm s}r}{2|s-t|}\right)\frac{J_1(\kappa_{\rm s}r)}{r}
-\kappa_{\rm p}\left(1+\frac{2\rm i}{\pi}C_E+\frac{2\rm i}{\pi}
\ln \frac{\kappa_{\rm p}r}{2|s-t|}\right)\frac{J_1(\kappa_{\rm p}r)}{r}\\ \label{f18}
&&\qquad\qquad-\frac{\rm i}{2\pi}\sum_{p=0}^{+\infty}\frac{(-1)^p}{p!(p+1)!}
\left(\frac{r}{2}\right)^{2p}(\kappa_{\rm s}^{2+2p}-\kappa_{\rm p}^{2+2p})(\phi(p)+\phi(p+1)),\\ \nonumber
&&\gamma_2(s,t)=\kappa_{\rm s}^2\left(1+\frac{2\rm i}{\pi}C_E+\frac{2\rm i}{\pi}
\ln \frac{\kappa_{\rm s}r}{2|s-t|}\right)J_2(\kappa_{\rm s}r)-\kappa_{\rm p}^2
\left(1+\frac{2\rm i}{\pi}C_E+\frac{2\rm i}{\pi}
\ln \frac{\kappa_{\rm p}r}{2|s-t|}\right)J_2(\kappa_{\rm p}r)\\ \label{f19}
&&\qquad\qquad-\frac{\rm i}{\pi}(\kappa_{\rm s}^2-\kappa_{\rm p}^2)
-\frac{\rm i}{\pi}\sum_{p=0}^{+\infty}\frac{(-1)^p}{p!(p+2)!}
\left(\frac{r}{2}\right)^{2+2p}(\kappa_{\rm s}^{4+2p}
-\kappa_{\rm p}^{4+2p})(\phi(p)+\phi(p+2)),\\ \nonumber
&&\gamma_3(s,t)=\kappa_{\rm s}^3r\left(1+\frac{2\rm i}{\pi}C_E
+\frac{2\rm i}{\pi}\ln \frac{\kappa_{\rm s}r}{2|s-t|}\right)J_3(\kappa_{\rm s}r)
-\kappa_{\rm p}^3r\left(1+\frac{2\rm i}{\pi}C_E+\frac{2\rm i}{\pi}
\ln \frac{\kappa_{\rm p}r}{2|s-t|}\right)J_3(\kappa_{\rm p}r)\\ \nonumber
&&\qquad\qquad-\frac{2\rm i}{\pi}(\kappa_{\rm s}^2-\kappa_{\rm p}^2)
-\frac{{\rm i}r}{4\pi}(\kappa_{\rm s}^4-\kappa_{\rm p}^4)\\ \label{f20}
&&\qquad\qquad-\frac{{\rm i}r}{\pi}\sum_{p=0}^{+\infty}\frac{(-1)^p}{p!(p+3)!}
\left(\frac{r}{2}\right)^{3+2p}(\kappa_{\rm s}^{6+2p}-\kappa_{\rm p}^{6+2p})(\phi(p)+\phi(p+3)),
\end{eqnarray}
and for $j,k=1,2$,
\begin{eqnarray}\nonumber
\sigma_{jk}(s,t)&=&\frac{1}{\pi}\frac{\lambda+\mu}{\lambda+3\mu}\xi(s,t)\delta_{jk}
+\frac{\rm  i}{2\mu}\left\{\widetilde{\lambda}(x_j(s)-y_j(t))l_k(t)
-\widetilde{\mu}\left[(f(s)-f(t))\delta_{j1}-(s-t)\delta_{j2}\right]l^{\perp}_k(t)\right\}\\ \nonumber
&&\times\Bigg\{\kappa_{\rm s}\left(1+\frac{2\rm i}{\pi}C_E
+\frac{2\rm i}{\pi}\ln\frac{\kappa_{\rm s}r}{2|s-t|}\right)\frac{J_1(\kappa_{\rm s}r)}{r}
-\frac{{\rm i}\kappa_{\rm s}^2}{2\pi}\sum_{p=0}^{+\infty}\frac{(-1)^p}{p!(p+1)!}
\left(\frac{\kappa_{\rm s}r}{2}\right)^{2p}(\phi(p)+\phi(p+1))\Bigg\}\\ \nonumber
&&-\frac{\rm i}{2\omega^2 }\left\{(\mu+\widetilde{\mu})\left[(x_j(s)-y_j(t))l_k(t)
+(x_k(s)-y_k(t))l_j(t)\right]+4\widetilde{\lambda}(x_j(s)-y_j(t))l_k(t)\right\}\\ \nonumber
&&\times\Bigg\{\kappa_{\rm s}^2\left(1+\frac{2\rm i}{\pi}C_E
+\frac{2\rm i}{\pi}\ln\frac{\kappa_{\rm s}r}{2|s-t|}\right)\frac{J_2(\kappa_{\rm s}r)}{r^2}
-\kappa_{\rm p}^2\left(1+\frac{2\rm i}{\pi}C_E+\frac{2\rm i}{\pi}
\ln\frac{\kappa_{\rm p}r}{2|s-t|}\right)\frac{J_2(\kappa_{\rm p}r)}{r^2}\\ \nonumber
&&-\frac{\rm i}{4\pi}\sum_{p=0}^{+\infty}\frac{(-1)^p}{p!(p+2)!}
\left(\frac{r}{2}\right)^{2p}(\kappa_{\rm s}^{4+2p}-\kappa_{\rm p}^{4+2p})(\phi(p)+\phi(p+2))\Bigg\}
+\frac{{\rm i}\widetilde{\lambda}}{2\omega^2}\frac{(x_j(s)-y_j(t))l_k(t)}{r}\\ \nonumber
&&\times\Bigg\{\kappa_{\rm s}^3\left(1+\frac{2\rm i}{\pi}C_E+\frac{2\rm i}{\pi}
\ln\frac{\kappa_{\rm s}r}{2|s-t|}\right)J_3(\kappa_{\rm s}r)-\kappa_{\rm p}^3\left(1+\frac{2\rm i}{\pi}C_E
+\frac{2\rm i}{\pi}\ln \frac{\kappa_{\rm p}r}{2|s-t|}\right)J_3(\kappa_{\rm p}r)\\ \label{f21}
&&-\frac{\rm i}{4\pi}(\kappa_{\rm s}^4-\kappa_{\rm p}^4)-\frac{\rm i}{\pi}\sum_{p=0}^{+\infty}\frac{(-1)^p}{p!(p+3)!}\left(\frac{r}{2}\right)^{3+2p}(\kappa_{\rm s}^{6+2p}
-\kappa_{\rm p}^{6+2p})(\phi(p)+\phi(p+3))\Bigg\},
\end{eqnarray}
where $\phi(0):=0$ and $\phi(p):=\sum_{m=1}^p\frac{1}{m}$ for $p=1,2,3,...$. By a straightforward 
calculation and \cite[Section 7.1.3]{KA}, we have $\frac{r}{|s-t|}\in BC^n(\R^2)$, 
$\sqrt{1+f'(t)^2}\in BC^n(\R)$, $\zeta_{jk}(s,t)\in BC^n(\R^2)$ and $\xi(s,t)\in BC^n(\R^2)$. 
Thus, using (\ref{f16})--(\ref{f20}), we conclude that $\rho_0(\kappa, s, t)\in BC^n(\R^2)$, 
$\kappa_{\rm s}r\rho_1(\kappa_{\rm s}, s, t)\in BC^n(\R^2)$, $\gamma_1(s,t)\in BC^n(\R^2)$, 
$\gamma_2(s,t)\in BC^n(\R^2)$, and $\gamma_3(s,t)\in BC^n(\R^2)$. By a careful observation 
from (\ref{f21}), we have $\sigma_{jk}(s,t)\in BC^n(\R^2)$.
Since $\bm{A}^{(3)}(s,t)=\bm{A}^{(4)}(s,t)-\bm{A}^{(5)}(s,t)$, it follows from the smoothness 
of $\bm{G}(x,y')$ and $\bm{U}(x,y)$ for $x,y\in U_h$ that $\bm{A}^{(3)}(s,t)$ is smooth. These, 
together with (\ref{f15}), imply that $C^*_{jk}(s,t)\in BC^{n}(\R^2)$. Thus, the condition ${\bm C2}$ holds.

For the condition ${\bm C3}$, it can be seen in \cite[Theorem 2.1]{TA02} that the elements of 
the Green's tensor $\bm{G}_{D,h}$ satisfies
\begin{eqnarray}\nonumber
 \max_{j,k=1,2}\left|G_{D,h,jk}(x,y)\right|\lesssim \frac{1+(x_2-h)(y_2-h)}{|x_1-y_1|^{3/2}}
\end{eqnarray}
for $x,y\in U_h$ and $|x_1-y_1|\geq \varepsilon>0$.
This, together with the regularity estimates for solutions to elliptic partial differential equations 
(see \cite[Theorem 3.9]{GT83}), implies that such estimates actually hold for partial derivatives of $\bm{G}_{D,h}(x,y)$ of any order. Note that $\bm{A}(s,t)$ is related to $\bm{G}_{D,h}(x,y)$ and its 
derivatives, thus the condition ${\bm C3}$ holds with $p=3/2$ and $C>0$ only depends on $M$.
The proof is thus complete.
\end{proof}
\begin{remark}\label{thm42}
	From the proof of Theorem \ref{thm41}, it easily follows that $\bm{B}(s,t)\in C_{0,\pi}^n(\R^2)$ and $\bm{C}(s,t)\in BC_p^n(\R^2)$ with $p=3/2$ if $f\in B_{c,M}^{(n)}$.
\end{remark}

\section{Numerical results}\label{sec5}\setcounter{equation}{0}

The purpose of this section is to illustrate the feasibility of the Nystr\"{o}m method by several 
numerical examples. As presented in (\ref{a6}), the expression of $\bm{G}_{D,h}(x,y)$ involves 
the matrix function $\bm{U}(x,y)$ which is smooth on the boundary $\Gamma$. However,  
$\bm{U}(x,y)$ is given in terms of an improper integral on the infinite interval which is difficult 
to compute numerically. Thus in the numerical experiments,
we replace $\bm{G}_{D,h}(x,y)$ arising in (\ref{a10}) by $\bm{G}(x,y)-\bm{G}(x,y')$ to
avoid the complicated computation of $\bm{U}(x,y)$. It is shown that the numerical experiments 
are indeed satisfactory by using this replacement.

In the following examples, we assume that the Lam\'{e} constants $\lambda=1$, $\mu=1$ and 
the frequency $\omega=20$. For the Nystr\"{o}m method of the integral equation (\ref{f2}), 
we choose $cut = 10\pi$. Setting $s=t_j$ for $j=0,1,2,...,2cut/h$ in (\ref{f6}) gives a linear 
system of equations which can be solved to obtain the density $\widetilde{\bm{\psi}}_N$, and 
then we  can calculate the solution
$\bm{u}^{\rm sc}$ through (\ref{a9}). In each example, we compute  the scattered field at  random points $z_i, i=1,...,Nb$ in the region $[-2.5,2.5]\times[0.5,1.5]$, where the number of random points $Nb=101$. See the blue points in Figure \ref{profile} for the geometry profile. The elastic scattered field is a vector that can be written as $\bm{u}^{\rm sc}=(u^{\rm sc}_1,u^{\rm sc}_2)$, we will compute the following error for this scattered field
\be\label{error}
E(v):=\frac{1}{Nb}{\sum_{i=1}^{Nb}|v(z_i)-v^{app}(z_i)|^2}.%{\Sigma_{i=1}^{Nb}|\tilde{u}|^2}
\en
Here $v$ is chosen to be Re $u^{\rm sc}_i$, Im $u^{\rm sc}_i$ or $\left|u^{\rm sc}_i\right|$ for $i=1,2$ in our numerical implementation, $v^{app}$ is the corresponding value computed by our Nystr\"{o}m method.

\begin{figure}[htbp]
	\centering
	\subfigure[flat plane]
	{\includegraphics[width=0.28\linewidth]{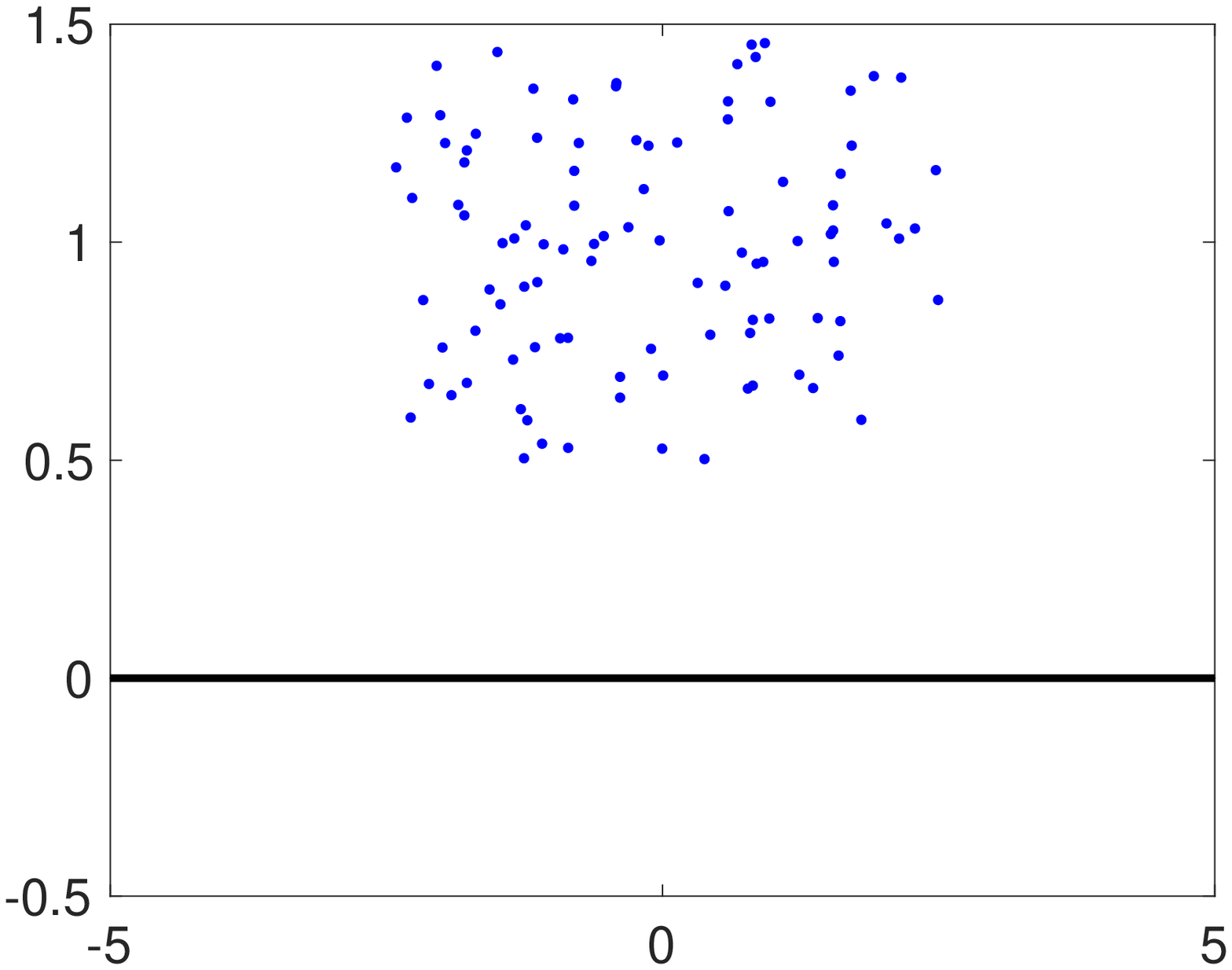}}
	\subfigure[periodic surface]
	{\includegraphics[width=0.28\linewidth]{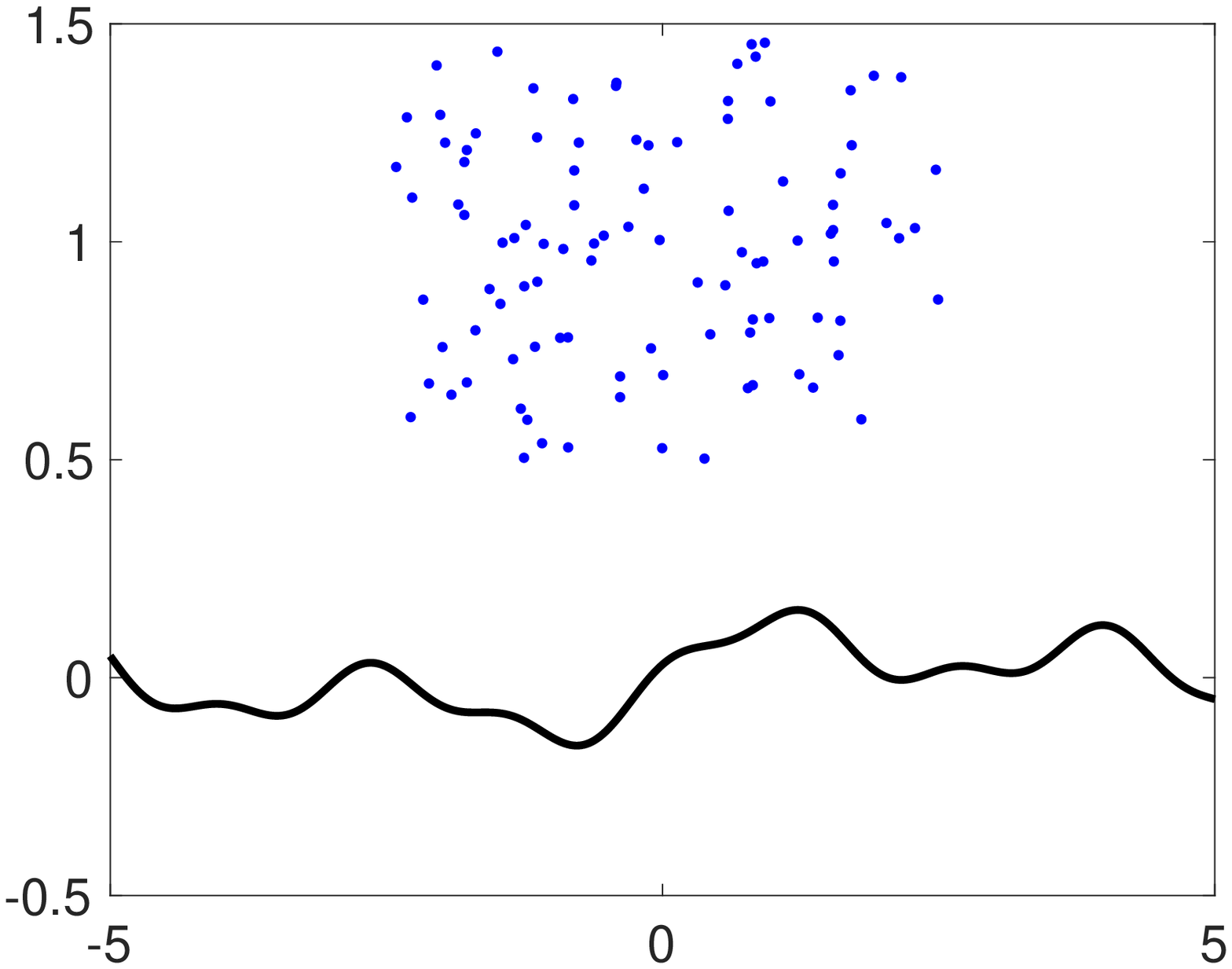}}
	\subfigure[rough surface]
	{\includegraphics[width=0.28\linewidth]{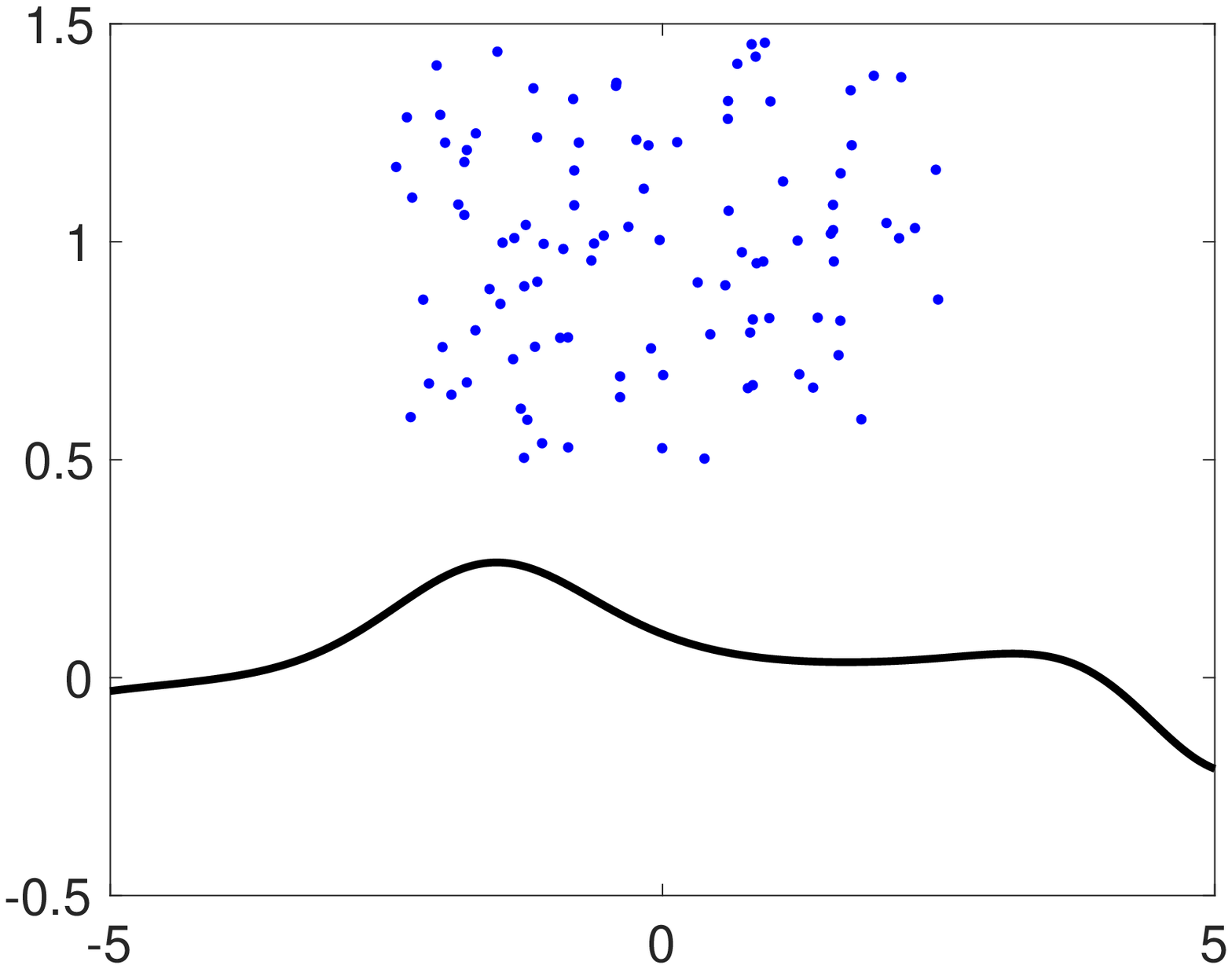}}
\caption{The solid lines in (a)--(c) represent the profile of the scattering interface for Examples 1--3, respectively, and the scattered field is computed on the blue random points in  the region $[-2.5,2.5]\times[0.5,1.5]$.} \label{profile}
\end{figure}

{\bf Example 1}. We consider the elastic scattering by a planar $x_2=0$ with an incident plane wave. 
The profile of the flat surface is given in Figure \ref{profile}(a). In general, an elastic plane 
wave can be written as a linear combination of a compressional plane wave 
$\bm{u}_{\rm p}^{\rm inc}(x;\theta)$ and a shear plane wave $\bm{u}_{\rm s}^{\rm inc}(x;\theta)$, that is,
\begin{eqnarray}\label{g1}
\bm{u}^{\rm inc}(x;\theta)=\alpha \bm{u}_{\rm p}^{\rm inc}(x;\theta)+\beta\bm{u}_{\rm s}^{\rm inc}(x;\theta),
\qquad \alpha, \beta\in\mathbb C,
\end{eqnarray}
where $\bm{u}_{\rm p}^{\rm inc}(x;\theta)= \theta e^{{\rm i}\kappa_{\rm p}x\cdot \theta}$ and  
$\bm{u}_{\rm s}^{\rm inc}(x;\theta)= \theta^{\perp} e^{{\rm i}\kappa_{\rm s}x\cdot \theta}$ 
with {$\theta\in{\mathbb S}:=\{x\in\R^2:|x|=1\}$} being an incident direction. In this example, 
we choose $\theta = (0,-1)^{\top}$. We consider the cases $(\alpha, \beta) = (1,0)$ 
and $(\alpha, \beta) = (0,1)$ in \eqref{g1}, then the corresponding incident waves are
\begin{eqnarray}\nonumber
\bm{u}^{\rm inc}_{\rm 1}(x) = \bm{u}^{\rm inc}_{\rm p}(x)=(0, -e^{-{\rm i}\kappa_{\rm p}x_2})^{\top}\quad\text{and}\quad
\bm{u}^{\rm inc}_{\rm 2}(x)=\bm{u}^{\rm inc}_{\rm s}(x)=(-e^{-{\rm i}\kappa_{\rm s}x_2}, 0)^{\top},
\end{eqnarray}
respectively. Since the rough surface is given by a planar $x_2=0$, it is easily seen that the 
corresponding scattered fields can be written explicitly as
\begin{eqnarray}\label{g3}
\bm{u}_{\rm 1}^{\rm sc}(x)=(u_{\rm 1, 1}^{\rm sc}, u_{\rm 1, 2}^{\rm sc})^{\top}=(0, e^{{\rm i}\kappa_{\rm p}x_2})^{\top}\quad\text{and}\quad
\bm{u}_{\rm 2}^{\rm sc}(x)=(u_{\rm 2, 1}^{\rm sc}, u_{\rm 2, 2}^{\rm sc})^{\top}=(e^{{\rm i}\kappa_{\rm s}x_2}, 0)^{\top},
\end{eqnarray}
respectively. Table \ref{table1} and Table \ref{table2} present the errors between the numerical results of the scattered fields and the exact solution computed by  \eqref{error} for the cases $(\alpha, \beta) = (1,0)$ and $(\alpha, \beta) = (0,1)$, respectively. 
In each table, we give the errors that calculated by our Nystr\"{o}m method with $N=8,16,32,64, 128$, 
respectively. It can be seen from these two tables that the error between the numerical solution and the exact solution  converges to $0$ as $N$ increases.

\begin{table}[htbp]
\footnotesize
\centering
\caption{Error against $N$ for the incident wave $\bm{u}^{\rm inc}_{\rm 1}(x)$ 
in Example 1 with a planar surface}
\begin{tabular}{c rrrrrr}
\hline
\multicolumn{1}{c }{N}
		& \multicolumn{1}{c}{E(Re $u^{\rm sc}_{\rm 1,1}$)}
		& \multicolumn{1}{c}{E(Im $u^{\rm sc}_{\rm 1,1}$)}
		& \multicolumn{1}{c}{E($|u^{\rm sc}_{\rm 1,1}|$)}
		& \multicolumn{1}{c}{E(Re $u^{\rm sc}_{\rm 1,2}$)}
		& \multicolumn{1}{c}{E(Im $u^{\rm sc}_{\rm 1,2}$)}
		& \multicolumn{1}{c}{E($|u^{\rm sc}_{\rm 1,2}|$)}\\ 
\hline
%1 & 0.0062279478&	0.0061870861&	0.0124150339&	0.5924773570&	0.4755013514&	1.0679787085\\
%2 & 0.0037321823&	0.0038754150&	0.0076075973&	0.6131780874&	0.4318626769&	1.0450407644\\	
%4& 0.0893014989&	0.0878784903&	0.1771799892&	0.6983922729&	0.5727078854&	1.2711001584\\
8 & 0.1706786080&	0.1604348338&	0.3311134418&	0.8032082301&	0.6041408954&	1.4073491255\\
16& 0.0001175089&	0.0000843658&	0.0002018747&	0.0113121237&	0.0090742350&	0.0203863586\\
32&	0.0002007818&	0.0001673652&	0.0003681471&	0.0003264799&	0.0000417871&	0.0003682670\\
64&	0.0001618967&	0.0001411445&	0.0003030412&	0.0001354001&	0.0000366766&	0.0001720767\\
128& 0.0001445112&	0.0001313396&	0.0002758508&	0.0000699852&	0.0000337565&	0.0001037417\\ 
\hline
\end{tabular}
\label{table1}
\end{table}	

\begin{table}[htbp]
\footnotesize
\centering
\caption{Error against $N$ for the incident wave $\bm{u}^{\rm inc}_{\rm 2}(x)$ 
in Example 1 with a planar surface}
\begin{tabular}{c rrrrrr}
\hline
\multicolumn{1}{c}{N}
		& \multicolumn{1}{c}{E({Re} $u^{\rm sc}_{\rm 2,1}$)}
		& \multicolumn{1}{c}{E(Im $u^{\rm sc}_{\rm 2,1}$)}
		& \multicolumn{1}{c}{E($|u^{\rm sc}_{\rm 2,1}|$)}
		& \multicolumn{1}{c}{E(Re $u^{\rm sc}_{\rm 2,2}$)}
		& \multicolumn{1}{c}{E(Im $u^{\rm sc}_{\rm 2,2}$)}
		& \multicolumn{1}{c}{E($|u^{\rm sc}_{\rm 2,2}|$)}\\ 
\hline
%1&	0.4910497667&	0.5229937140&	1.0140434807&	0.0305315470&	0.0082922243&	0.0388237713\\
%2&	0.4660292771& 0.4819285478&	0.9479578249&	0.5098649766&	0.0644970471&	0.5743620236\\
%4& 0.4373636206&	0.4658204412&	0.9031840618&	0.0664333183&	0.0609604389&	0.1273937572\\	
8&	0.2047793681&	0.2113485590&	0.4161279271&	0.3438342461&	0.3232904684&	0.6671247145\\
16&	0.0000753224&	0.0001141582&	0.0001894806&	0.0000774519&	0.0000768715&	0.0001543233\\
32&	0.0000597451&	0.0000798512&	0.0001395963&	0.0000216754&	0.0000316223&	0.0000532977\\
64&	0.0000557738&	0.0000576602&	0.0001134340&	0.0000130030&	0.0000180143&	0.0000310173\\
128& 0.0000526921&	0.0000485121&	0.0001012041&	0.0000134101&	0.0000154517&	0.0000288618\\ 
\hline
\end{tabular}
\label{table2}
\end{table}	

{\bf Example 2}. We consider the elastic scattering by a periodic unbounded rough surface with
the periodic surface given by
\begin{eqnarray}\nonumber
f(x_1)=0.084\sin(0.6\pi x_1)+0.084\sin(0.24\pi x_1)+0.03\sin(1.5\pi (x_1-1)).
\end{eqnarray}
See the profile of this periodic surface in Figure \ref{profile}(b). The incident wave is chosen to be
\be\label{lxl1}
{\bm u}^{\rm inc}_{\rm 3}(x) = \bm{G}(x,z){\bm q}
\en
with the point $z=(0,-3)$ and the polarization direction ${\bm q}=(0.6,0.8)^{\top}$. Due to the fact that the point $z$ is below the surface $\Gamma$, it follows from the well-posedness of the problem ({\bf DP}) that the corresponding scattered field has the explicit expression
\begin{eqnarray}\label{g5}
\bm{u}^{\rm sc}_{\rm 3}(x) =(u_{\rm 3, 1}^{\rm sc}, u_{\rm 3, 2}^{\rm sc})^{\top}= -\bm{G}(x,z){\bm q},\qquad x\in\Omega.
\end{eqnarray}
Table \ref{table3} gives the errors between the numerical results of the scattered field calculated by our Nystr\"{o}m method with $N=8,16,32,64,128$ and the exact solution, respectively. It can be seen from Table \ref{table3} that our Nystr\"{o}m method provides a satisfactory numerical results for this case.

\begin{table}[htbp]
\footnotesize
\centering
\caption{Error against $N$ for the incident wave
		${\bm u}^{\rm inc}_{\rm 3}(x)$ in Example 2 with a periodic surface}
\begin{tabular}{c rrrrrr}
\hline
\multicolumn{1}{c}{N}
		& \multicolumn{1}{c}{E(Re $u^{\rm sc}_{3,1}$)}
		& \multicolumn{1}{c}{E(Im $u^{\rm sc}_{3,1}$)}
		& \multicolumn{1}{c}{E($|u^{\rm sc}_{3,1}|$)}
		& \multicolumn{1}{c}{E(Re $u^{\rm sc}_{3,2}$)}
		& \multicolumn{1}{c}{E(Im $u^{\rm sc}_{3,2}$)}
		& \multicolumn{1}{c}{E($|u^{\rm sc}_{3,2}|$)}\\ 
\hline
8&	0.0000729326&	0.0000683190&	0.0001412516&	0.0000460758&	0.0000434141&	0.0000894899\\
16&	0.0000186844&	0.0000208419&	0.0000395263&	0.0000323776&	0.0000350254&	0.0000674030\\
32&	0.0000000333&	0.0000000331&	0.0000000664&	0.0000000159&	0.0000000179&	0.0000000339\\	
64&	0.0000000344&	0.0000000334&	0.0000000678&	0.0000000160&	0.0000000173&	0.0000000333\\							
128& 0.0000000346&	0.0000000335&	0.0000000681&	0.0000000161&	0.0000000173&	0.0000000333\\ 
\hline
\end{tabular}
\label{table3}
\end{table}	

{\bf Example 3}. We consider the elastic scattering by a non-periodic unbounded rough surface given by 
\begin{eqnarray}\nonumber
f(x_1)=0.1\cos(0.1x_1^2)e^{-\sin(x_1)}.
\end{eqnarray}
See the profile of this rough surface in Figure \ref{profile}(c). We choose the same incident wave ${\bm u}^{\rm inc}_{\rm 3}(x)$ as in \eqref{lxl1}. Similar as in Example 2, the corresponding scattered field has the form \eqref{g5}. Table \ref{table4} presents the error between our numerical results and the exact solution.

\begin{table}[htbp]
\footnotesize
\centering
\caption{Error against $N$ for the incident fields
		${\bm u}^{\rm inc}_{\rm 3}(x)$ in Example 3 with a non-periodic surface}
\begin{tabular}{c rrrrrr}
\hline
\multicolumn{1}{c}{N}
		& \multicolumn{1}{c}{E(Re $u^{\rm sc}_{3,1}$)}
		& \multicolumn{1}{c}{E(Im $u^{\rm sc}_{3,1}$)}
		& \multicolumn{1}{c}{E($|u^{\rm sc}_{3,1}|$)}
		& \multicolumn{1}{c}{E(Re $u^{\rm sc}_{3,2}$)}
		& \multicolumn{1}{c}{E(Im $u^{\rm sc}_{3,2}$)}
		& \multicolumn{1}{c}{E($|u^{\rm sc}_{3,2}|$)}\\ 
\hline
8& 0.0000448509&	0.0000367475&	0.0000815984& 0.0000443602&	0.0000440444&	0.0000884046\\																				16&	0.0000098770&	0.0000073343&	0.0000172113&	0.0000107694&	0.0000074233&	0.0000181927\\							
32&	0.0000001459&	0.0000001046&	0.0000002505&	0.0000000564&	0.0000000511&	0.0000001075\\									
64&	0.0000001320&	0.0000000788&	0.0000002108& 0.0000000510&	0.0000000474&	0.0000000985\\
128& 0.0000001320&	0.0000000782&	0.0000002102& 0.0000000506&	0.0000000479&	0.0000000985\\							 \hline
\end{tabular}
\label{table4}
\end{table}	
%
%The above numerical examples show the effectiveness of our method. It can be seen that the scattered field computed by the Nystr\"{o}m method converges to the exact solution as the number of points $N$ becomes larger and larger.

\section{Conclusion}\label{sec6}

In this paper, we present a Nystr\"{o}m method for the two-dimensional time-harmonic elastic scattering by unbounded rough surfaces with Dirichlet boundary condition. With the aid of the ascending series expansions of the Bessel functions, we analyze the singularities of the relevant integral kernels. Based on this, we obtain the superalgebraic convergence rate of the Nystr\"{o}m method depending on the smoothness of the rough surfaces.
Several numerical examples demonstrate that the numerical solution converges when the number of quadrature points $N$ on the rough surface increases.
A possible continuation is to extend the present work to the case of impedance boundary condition
or the case of penetrable interface, which will be our future work.

\section*{Acknowledgements}

This work was partially supported by the National Key R \& D Program of China (2018YFA0702502),
the NNSF of China grant 11871466 and Youth Innovation Promotion Association CAS.

\end{document}